\documentclass[10pt]{article}
\usepackage{latexsym,amsfonts,amssymb,amsmath,amsthm}
\usepackage{graphicx}

\usepackage{amsmath}

\usepackage[usenames,dvipsnames]{color}
\usepackage{ulem}

\begin{document}
\setlength{\baselineskip}{14pt}

\parindent 0.5cm
\evensidemargin 0cm \oddsidemargin 0cm \topmargin 0cm \textheight 22cm \textwidth 16cm \footskip 2cm \headsep
0cm

\newtheorem{theorem}{Theorem}[section]
\newtheorem{lemma}{Lemma}[section]
\newtheorem{proposition}{Proposition}[section]
\newtheorem{definition}{Definition}[section]
\newtheorem{example}{Example}[section]
\newtheorem{corollary}{Corollary}[section]

\newtheorem{remark}{Remark}[section]

\numberwithin{equation}{section}

\def\p{\partial}
\def\I{\textit}
\def\R{\mathbb R}
\def\C{\mathbb C}
\def\u{\underline}
\def\l{\lambda}
\def\a{\alpha}
\def\O{\Omega}
\def\e{\epsilon}
\def\ls{\lambda^*}
\def\D{\displaystyle}
\def\wyx{ \frac{w(y,t)}{w(x,t)}}
\def\imp{\Rightarrow}
\def\tE{\tilde E}
\def\tX{\tilde X}
\def\tH{\tilde H}
\def\tu{\tilde u}
\def\d{\mathcal D}
\def\aa{\mathcal A}
\def\DH{\mathcal D(\tH)}
\def\bE{\bar E}
\def\bH{\bar H}
\def\M{\mathcal M}
\renewcommand{\labelenumi}{(\arabic{enumi})}

\def\disp{\displaystyle}
\def\undertex#1{$\underline{\hbox{#1}}$}
\def\card{\mathop{\hbox{card}}}
\def\sgn{\mathop{\hbox{sgn}}}
\def\exp{\mathop{\hbox{exp}}}
\def\OFP{(\Omega,{\cal F},\PP)}
\newcommand\JM{Mierczy\'nski}
\newcommand\RR{\ensuremath{\mathbb{R}}}
\newcommand\CC{\ensuremath{\mathbb{C}}}
\newcommand\QQ{\ensuremath{\mathbb{Q}}}
\newcommand\ZZ{\ensuremath{\mathbb{Z}}}
\newcommand\NN{\ensuremath{\mathbb{N}}}
\newcommand\PP{\ensuremath{\mathbb{P}}}
\newcommand\abs[1]{\ensuremath{\lvert#1\rvert}}

\newcommand\normf[1]{\ensuremath{\lVert#1\rVert_{f}}}
\newcommand\normfRb[1]{\ensuremath{\lVert#1\rVert_{f,R_b}}}
\newcommand\normfRbone[1]{\ensuremath{\lVert#1\rVert_{f, R_{b_1}}}}
\newcommand\normfRbtwo[1]{\ensuremath{\lVert#1\rVert_{f,R_{b_2}}}}
\newcommand\normtwo[1]{\ensuremath{\lVert#1\rVert_{2}}}
\newcommand\norminfty[1]{\ensuremath{\lVert#1\rVert_{\infty}}}

\title{On Principal Spectrum Points/Principal Eigenvalues of Nonlocal Dispersal
Operators and Applications}

\author{
Wenxian Shen\thanks{Partially supported by NSF grant DMS--0907752} \,\, and\,\, Xiaoxia Xie \\
Department of Mathematics and Statistics\\
Auburn University\\
Auburn University, AL 36849\\
U.S.A. }

\date{}
\maketitle

\noindent {\bf Abstract.}
This paper is to investigate the dependence of the principal spectrum points of nonlocal dispersal operators on underlying parameters and to consider its applications.
In particular, we study the effects of the spatial inhomogeneity, the dispersal rate, and the dispersal distance on the existence of the principal eigenvalues,
the magnitude of the principal spectrum points, and the asymptotic behavior of the principal spectrum points of nonlocal dispersal operators with Dirichlet type,
 Neumann type, and periodic boundary conditions in a unified way. We also discuss the applications of the principal spectral theory of nonlocal dispersal operators
  to the asymptotic dynamics of two species competition systems with nonlocal dispersal.

\bigskip

\noindent {\bf Key words.} Nonlocal dispersal operator,  principal spectrum point, principal eigenvalue, spatial inhomogeneity, dispersal rate, dispersal distance, two species
competition system.
\bigskip

\noindent {\bf Mathematics subject classification.} 45C05, 45M05, 45M20, 47G10, 92D25.

\newpage

\section{Introduction}
\setcounter{equation}{0}

The current paper is devoted to the study of principal spectrum of the following three eigenvalue
problems associated to nonlocal dispersal operators,
\begin{equation}
\label{dirichlet-op} \nu_1\left[\int_D k(y-x)u(y)dy-u(x)\right]+a_1(x)u=\lambda
u(x),\quad\quad \quad x\in \bar D,
\end{equation}
where $D\subset \RR^N$ is a smooth bounded domain,
\begin{equation}
\label{neumann-op} \nu_2\int_D k(y-x)[u(y)-u(x)]dy+a_2(x)u(x)=\lambda
u(x),\quad\quad \,     x  \in \bar D,
\end{equation}
where $D\subset \RR^N$ is as in \eqref{dirichlet-op},
and
\begin{equation}
\label{periodic-op}
\begin{cases}
\nu_3[\int_{\RR^N} k(y-x)u(y)dy-u(x)]+a_3(x)u(x)=\lambda u(x),\quad
&x\in \RR^N, \cr u(x+p_j{\bf e_j})=u(x),\quad & x\in \RR^N,
\end{cases}
\end{equation}
where $p_j>0$, ${\bf e_j}=(\delta_{j1},\delta_{j2},\cdots,\delta_{jN})$
with $\delta_{jk}=1$ if $j=k$ and $\delta_{jk}=0$ if $j\not =k$, and
$a_3(x+p_j{\bf e_j})=a_3(x)$, $j=1,2,\cdots,N$. In  \eqref{dirichlet-op}, \eqref{neumann-op}, and \eqref{periodic-op},
$k(\cdot)$ is a nonnegative $C^1$ function with compact support, $k(0)>0$, and $\int_{\RR^N} k(z)dz=1$.

Observe that the nonlocal dispersal operators in \eqref{dirichlet-op}, \eqref{neumann-op}, and \eqref{periodic-op},  that is,
$u(x)\mapsto \int_D k(y-x)u(y)dy-u(x)$, $u(x)\mapsto \int_D k(y-x)[u(y)-u(x)]dy$, and $u(x)\mapsto \int_{\RR^N} k(y-x)u(y)dy-u(x)$,
 can be viewed as
$u(x)\mapsto \int_{\RR^N}k(y-x)[u(y)-u(x)]dy$ with Dirichlet type boundary condition $\int_{\RR^N\setminus D}k(y-x)u(y)dy=0$ for
$x\in\bar D$,
$u(x)\mapsto \int_{\RR^N} k(y-x)[u(y)-u(x)]dy$ with Neumann type boundary condition $\int_{\RR^N\setminus D}k(y-x)[u(y)-u(x)]dy =0$ for $x\in\bar D$,
and $u(x)\mapsto \int_{\RR^N} k(y-x)[u(y)-u(x)]dy$ with periodic boundary condition $u(x+p_j{\bf e_j})=u(x)$ for $x\in\RR^N$,  respectively.

Observe also that
the eigenvalue problems \eqref{dirichlet-op}, \eqref{neumann-op}, and \eqref{periodic-op} can be viewed as  the nonlocal
 counterparts of the following eigenvalue problems associated to random dispersal operators,
\begin{equation}
\label{random-dirichlet-op}
\begin{cases}
\nu_1\Delta u(x)+a_1(x)u(x)=\lambda u(x), \quad & x\in D, \cr u(x)=0, \quad & x\in \p D,
\end{cases}
\end{equation}
\begin{equation}
\label{random-neumann-op}
\begin{cases}
\nu_2\Delta u(x)+a_2(x)u(x)=\lambda u(x),\quad & x\in D, \cr \frac{\p u}{\p
n}(x)=0, \quad & x\in \p D,
\end{cases}
\end{equation}
and
\begin{equation}
\label{random-periodic-op}
\begin{cases}
\nu_3\Delta u(x)+a_3(x)u(x)=\lambda u(x), \quad & x\in \RR^N, \cr u(x+p_j{\bf
e_j})=u(x), \quad & x\in \RR^N,
\end{cases}
\end{equation}
respectively.   In \cite{ShXi1}, the authors of the current paper will explore the relations between \eqref{dirichlet-op} and \eqref{random-dirichlet-op}
(resp. \eqref{neumann-op} and \eqref{random-neumann-op}, \eqref{periodic-op} and \eqref{random-periodic-op})  and prove that
the principal eigenvalues of \eqref{random-dirichlet-op}, \eqref{random-neumann-op},
and \eqref{random-periodic-op}  can be approximated
by the principal spectrum points of \eqref{dirichlet-op}, \eqref{neumann-op}, and \eqref{periodic-op}
with properly rescaled kernels, respectively  (see Definition \ref{pev-def} for the definition of
principal spectrum points of  \eqref{dirichlet-op}, \eqref{neumann-op}, and \eqref{periodic-op}).

The nonlocal dispersal operator $u(x)\mapsto \int_{\RR^N} k(y-x)[u(y)-u(x)]dy$ with Dirichlet type or Neumann type or periodic  boundary condition and
the random dispersal operator $u(x)\mapsto \Delta u(x)$ with  Dirichlet or Neumann or periodic  boundary condition are  widely used to model diffusive systems in
applied sciences.  In particular, the random dispersal operator $u(x)\mapsto \Delta u(x)$ with proper boundary condition is usually adopted when the organisms
in a diffusive system move randomly between the adjacent spatial locations.  Nonlocal dispersal operator such
as $u(x) \mapsto  \int_{\RR^N}k(y-x)[u(y)-u(x)]dy$ is applied when
 diffusive
systems  exhibit long range internal interactions (see \cite{Fif}, \cite{GrHiHuMiVi}, \cite{HuMaMiVi}). Here if there is $\delta>0$ such that
${\rm supp}(k(\cdot))\subset B(0,\delta):=\{z\in\RR^N\,|\, \|z\|<\delta\}$ and for any $0<\tilde\delta<\delta$, ${\rm supp}(k(\cdot))\cap \Big( B(0,\delta)\setminus B(0,\tilde\delta)\Big)\not
=\emptyset$, $\delta$ is called  the {\it  dispersal distance} of the nonlocal dispersal operator  $u(x)\mapsto  \int_{\RR^N}k(y-x)[u(y)-u(x)]dy$.
  As a basic technical tool
for the study of nonlinear evolution equations with random and
nonlocal dispersals, it is of great importance to investigate
aspects of  spectral theory for random and nonlocal dispersal
operators.

The eigenvalue problems \eqref{random-dirichlet-op}, \eqref{random-neumann-op}, and \eqref{random-periodic-op},
and in particular, their associated principal eigenvalue problems, are well understood. For example,
it is known that the largest real part, denoted by $\lambda_{R,1}(\nu_1,a_1)$, of the spectrum set of \eqref{random-dirichlet-op}
 is an isolated algebraically simple eigenvalue of \eqref{random-dirichlet-op}
with a positive eigenfunction, and for any other  $\lambda$ in the spectrum set  of \eqref{random-dirichlet-op},
${\rm Re}\lambda<\lambda_{R,1}(\nu_1,a_1)$ ($\lambda_{R,1}(\nu_1,a_1)$ is called the {\it principal eigenvalue} of \eqref{random-dirichlet-op}).
Similar properties hold for the largest real parts, denoted by $\lambda_{R,2}(\nu_2,a_2)$ and $\lambda_{R,\nu_3}(\nu_3,a_3)$, of the
spectrum sets of
\eqref{random-neumann-op} and \eqref{random-periodic-op}.

The principal eigenvalue problems \eqref{dirichlet-op}, \eqref{neumann-op}, and \eqref{periodic-op} have also been  studied recently  by many
people  (see \cite{Cov}, \cite{GaRo}, \cite{HeShZh}, \cite{KaLoSh1},
\cite{ShZh0}, \cite{ShXi1}, and references therein).
Let $\tilde\lambda_1(\nu_1,a_1)$ (resp. $\tilde\lambda_2(\nu_2,a_2)$, $\tilde\lambda_3(\nu_3,a_3)$)  be the largest real part of the spectrum set of \eqref{dirichlet-op}
(resp. \eqref{neumann-op}, \eqref{periodic-op}).  $\tilde\lambda_1(\nu_1,a_1)$ (resp. $\tilde\lambda_2(\nu_2,a_2)$, $\tilde\lambda_3(\nu_3,a_3)$) is called the
{\it principal spectrum point} of \eqref{dirichlet-op} (resp. \eqref{neumann-op}, \eqref{periodic-op}). $\tilde\lambda_1(\nu_1,a_1)$ (resp. $\tilde\lambda_2(\nu_2,a_2)$, $\tilde\lambda_3(\nu_3,a_3)$) is also called the
{\it principal eigenvalue} of \eqref{dirichlet-op} (resp. \eqref{neumann-op}, \eqref{periodic-op}) if it is an isolated algebraically simple
 eigenvalue of  \eqref{dirichlet-op} (resp. \eqref{neumann-op}, \eqref{periodic-op})
with a positive eigenfunction (see Definition \ref{pev-def} for detail).
It is known that a nonlocal dispersal operator may not have a principal
eigenvalue (see \cite{Cov}, \cite{ShZh0} for examples), which reveals some essential difference between nonlocal and random dispersal operators.
Some sufficient conditions are provided in \cite{Cov}, \cite{KaLoSh1},  and \cite{ShZh0} for the existence of principal eigenvalues of \eqref{dirichlet-op}, \eqref{neumann-op},
 and \eqref{periodic-op} (the conditions in \cite{Cov} apply to \eqref{dirichlet-op} and \eqref{neumann-op}, the conditions in \cite{KaLoSh1} apply to \eqref{dirichlet-op},
 and the conditions in \cite{ShZh0} apply to \eqref{periodic-op}). Such sufficient conditions have been found important in the study
 of nonlinear evolution equations with nonlocal dispersals (see \cite{Cov}, \cite{HeNgSh}, \cite{HeShZh}, \cite{KaLoSh1},
 \cite{KaLoSh2}, \cite{KoSh}, \cite{ShZh0}, \cite{ShZh1}, \cite{ShZh2}).
However, the understanding is still little to many interesting questions regarding the principal spectrum points/principal eigenvalues of nonlocal dispersal
operators, including the dependence of principal spectrum points or principal eigenvalues
(if exist) of nonlocal dispersal operators on the underlying parameters.

The  objective of the current paper is to investigate the dependence of the principal spectrum points of nonlocal dispersal operators  on the underlying parameters.
In particular, we  study the effects of the spatial inhomogeneity, the dispersal rate, and the dispersal distance on the existence of principal eigenvalues, on the magnitude  of the principal spectrum points, and  on the asymptotic behavior of the principal spectrum points of nonlocal dispersal operators with different types of boundary conditions
in a unified way. Among others, we obtain the following:

\smallskip

\noindent $\bullet$  {\it criteria for $\tilde\lambda_1(\nu_1,a_1)$ (resp. $\tilde\lambda_2(\nu_2,a_2)$, $\tilde\lambda_3(\nu_3,a_3)$)
to be the principal eigenvalue of \eqref{dirichlet-op} (resp. \eqref{neumann-op}, \eqref{periodic-op}) } (see Theorem \ref{spatial-effect-thm} (1), (2),
Theorem \ref{dispersal-rate-effect-thm} (3), and Theorem  \ref{dispersal-distance-effect-thm} (3) for detail);

\medskip

\noindent $\bullet$ {\it  lower bounds of $\tilde\lambda_i(\nu_i,a_i)$  in terms of $\hat a_i$, where $\hat a_i$ is the spacial average of $a_i(x)$ ($i=2,3$)}  (see  Theorem \ref{spatial-effect-thm}
(4) for detail);

\medskip

\noindent $\bullet$ {\it monotonicity of $\tilde\lambda_i(\nu_i,a_i)$ with respect to $a_i(x)$ and  $\nu_i$ ($i=1,2,3$)}  (see Theorem \ref{spatial-effect-thm} (5) and Theorem \ref{dispersal-rate-effect-thm} (1) for detail);

\medskip

\noindent $\bullet$ {\it  limits of $\tilde\lambda_i(\nu_i,a_i)$ as $\nu_i\to 0$ and $\nu_i\to\infty$ ($i=1,2,3$) } (see Theorem \ref{dispersal-rate-effect-thm} (4), (5) for detail);

\medskip

\noindent $\bullet$ {\it limits of $\tilde\lambda_i(\nu_i,a_i,\delta)$ as $\delta\to 0$ and $\delta\to\infty$ in the case $k(z)=\frac{1}{\delta^N}\tilde k(\frac{z}{\delta})$ and
$\tilde k(z)\ge 0$, ${\rm supp}(\tilde k)=B(0,1)$, $\int_{\RR^N}\tilde k(z)dz=1$, where $\tilde\lambda_i(\nu_i,a_i,\delta)=\tilde\lambda(\nu_i,a_i)$ ($i=1,2,3$)}
 (see  Theorem  \ref{dispersal-distance-effect-thm}
(1), (2) for detail).

\medskip

\noindent We also investigate the applications of   principal spectrum point properties of nonlocal dispersal operators to the asymptotic dynamics of the following
two species competition system,
\begin{equation}
\label{competition-system-eq}
\begin{cases}
u_t=\nu[\int_D k(y-x)u(t,y)dy -u(t,x)]+u f(x,u+v),\quad x\in\bar D,\cr
v_t=\nu \int_D k(y-x)[u(t,y)-u(t,x)]dy +v f(x,u+v),\quad x\in \bar D,
\end{cases}
\end{equation}
where $D$ and $k(\cdot)$ are as in \eqref{dirichlet-op} with $k(-z)=k(z)$  and $f(\cdot,\cdot)$ is a $C^1$ function satisfying that $\tilde \lambda_1(\nu,f(\cdot,0))>0$, $f(x,w)<0$ for $w\gg 1$, and
$\p_2 f(x,w)<0$ for $w>0$. \eqref{competition-system-eq} models the population dynamics of two competing species with the same local population dynamics
(i.e. the same growth rate function  $f(\cdot,\cdot)$), the same dispersal rate (i.e. $\nu$), but one species
 adopts nonlocal dispersal with Dirichlet type boundary condition and the other adopts nonlocal dispersal with Neumann type boundary condition,
 where  $u(t,x)$ and $v(t,x)$ are the population
densities of two species at time $t$ and space location $x$. We show

\medskip

\noindent $\bullet$ the species diffusing nonlocally  with Neumann type boundary condition drives the species diffusing nonlocally with Dirichlet type boundary condition extinct (see Theorem \ref{competition-system-thm} for detail).
\medskip

Nonlocal evolution equations have been attracting more and more attentions due to the presence
of nonlocal interactions in many diffusive systems in applied sciences.
The reader is referred to \cite{BaCh},  \cite{BaZh}, \cite{ChChRo},  \cite{Che}, \cite{CoElRoWo}, \cite{CoElRo}, \cite{Cov1},
\cite{CoDu}, \cite{CoDaMa}, \cite{GaRo}, \cite{HuShVi},
\cite{KaLoSh1}, \cite{KaLoSh2}, \cite{LiSuWa}, \cite{LvWa}, \cite{PaLiLi}, \cite{ShVi}, etc.  for the study of various aspects of nonlocal dispersal equations.

The rest of this paper is organized as follows. In section 2, we  introduce some standing notations, definitions, and state the main results of the paper (i.e.
Theorems \ref{spatial-effect-thm}-\ref{competition-system-thm}).
In section 3, we present some preliminary materials to be used in the proofs of the main theorems in later sections.
We investigate the effects of spatial variation on the principal spectrum points of nonlocal dispersal operators and prove Theorem  \ref{spatial-effect-thm}
in section 4. In section 5, we consider the effects of dispersal rate   on the principal spectrum points of nonlocal dispersal operators and prove Theorem  \ref{dispersal-rate-effect-thm}.
In  section 6, we explore the effects of dispersal distance on the principal spectrum points of nonlocal dispersal operators and prove Theorem  \ref{dispersal-distance-effect-thm}.
In the last section, we consider the asymptotic dynamics of \eqref{competition-system-eq} by applying some of the principal spectrum point properties of nonlocal dispersal
operators and prove Theorem \ref{competition-system-thm}.

\section{Notations, Definitions and Main Results}

Let
\begin{equation}
\label{x-d-space}
X_i=C(\bar D)
\end{equation}
with norm $\|u\|_{X_i}=\max_{x\in\bar D}|u(x)|$ for $i=1,2$,
\begin{equation}
\label{x-d-positive-cone}
X_i^+=\{u\in X_i\,|\, u(x)\geq 0,\quad x\in\bar D\},\quad i=1,2,
\end{equation}
and \begin{equation}
\label{x-d-positive-interior}
X_i^{++}={\rm Int} (X_i^+)=\{u\in X_i^+\,|\, u(x)>0,\quad x\in \bar D\},\quad i=1,2.
\end{equation}
Let
\begin{equation}
\label{x-p-space}
X_3=\{u\in C(\RR^N,\RR)\,|\, u(x+p_j{\bf e_j})=u(x),\quad x\in\RR^N,j=1,2,\cdots,N\}
\end{equation}
with norm $\|u\|_{X_3}=\max_{x\in\RR^N}|u(x)|$,
\begin{equation}
\label{x-p-positive-cone}
X_3^+=\{u\in X_3\,|\, u(x)\geq 0,\quad x\in\RR^N\},
\end{equation}
and
\begin{equation}
\label{x-p-positive-interior}
X_3^{++}={\rm Int}(X_3^+)=\{u\in X_3^+\,|\, u(x)>0,\quad x\in\RR^N\}.
\end{equation}

Let
\begin{equation}
\label{K-D-op}
\mathcal{K}_i: X_i\to X_i,\,\, (\mathcal{K}_iu)(x)=\int_D k(y-x)u(y)dy\quad \forall u\in X_i,\quad i=1,2,
\end{equation}
and
\begin{equation}
\label{K-P-op}
\mathcal{K}_3:X_3\to X_3,\,\, (\mathcal{K}_3u)(x)=\int_{\RR^N}k(y-x)u(y)dy\quad\forall u\in X_3.
\end{equation}

Observe that $X_2=X_1$ and $\mathcal{K}_2=\mathcal{K}_1$. The introduction of $X_2$ and $\mathcal{K}_2$ is for convenience.
Throughout the paper, $\mathcal{I}$ denotes the identity map in the space under consideration.

Let
\begin{equation}
\label{h-i-eq}
\begin{cases}
h_1(x)=-\nu_1+a_1(x),\cr h_2(x)=-\nu_2\int_D k(y-x)dy+a_2(x),\cr
h_3(x)=-\nu_3+a_3(x).
\end{cases}
\end{equation}
So, we have
\begin{equation}
\label{K-P-op}
h_i(\cdot)\mathcal {I}:X_i\rightarrow X_i, (h_i(\cdot)\mathcal{I}u)(x)=h_i(x)u(x) \quad \forall u\in X_i,\quad i=1,2, 3,
\end{equation}
where  $a_i\in X_i$, $i=1,2,3$.
And $a_i(\cdot)\mathcal I$ has the same meaning as in \eqref{K-P-op} with
$h_i(\cdot)$ being replaced by $a_i(\cdot)$.

In the following, for \eqref{periodic-op}, we put
\begin{align}
\label{periodic-domain-eq}
D=[0,p_1]\times [0,p_2]\times\cdots\times  [0,p_N].
\end{align}
For given $a_i\in X_i$, let
\begin{equation}
\label{spacial-average}
\hat a_i=\frac{1}{|D|}\int_D a_i(x)dx, \,\ i=1, 2, 3,
\end{equation}
where $|D|$ is the Lebesgue measure of $D$.
Let
$$
a_{i,\max}=\max_{x\in\bar D}a_i(x),\quad a_{i,\min}=\min_{x\in\bar D} a_i(x),
$$
and
$$
h_{i,\max}=\max_{x\in\bar D}h_i(x),\quad h_{i,\min}=\min_{x\in\bar D} h_i(x).
$$

Let $\sigma(\nu_i\mathcal{K}_i+h_i(\cdot)\mathcal{I})$  be the
spectrum of $\nu_i\mathcal{K}_i+h_i(\cdot)\mathcal{I}$ for $i=1,2,3$
and
\begin{equation} \label{tilde-lambda-eq}
 \tilde \lambda_i(\nu_i,a_i)=\sup\{{\rm Re}\mu\,|\,
\mu\in \sigma(\nu_i\mathcal{K}_i+h_i(\cdot)\mathcal{I})\},\quad
i=1,2,3.
\end{equation}

\begin{definition}
\label{pev-def}
Let $1\le i\le 3$ be given.
\begin{itemize}
\item[(1)] $\tilde\lambda_i(\nu_i,a_i)$ defined in \eqref{tilde-lambda-eq} is called the  {\rm principal spectrum point} of
$\nu_i\mathcal{K}_i+h_i(\cdot)\mathcal{I}$.

\item[(2)]   A real number
$\lambda_i(\nu_i,a_i)\in \RR$  is called the {\rm principal
eigenvalue} of $\nu_i\mathcal{K}_i+h_i(\cdot)\mathcal{I}$  if it is
an isolated algebraically simple eigenvalue of
$\nu_i\mathcal{K}_i+h_i(\cdot)\mathcal{I}$ with a positive
eigenfunction and for any $\mu\in
\sigma(\nu_i\mathcal{K}_i+h_i(\cdot)\mathcal{I})\setminus\{\lambda_i(\nu_i,a_i)\}$,
${\rm Re}\mu <\lambda_i(\nu_i,a_i)$.
\end{itemize}
\end{definition}

Observe that   $\tilde\lambda_i(\nu_i,a_i)\in
\sigma(\nu_i\mathcal{K}_i+h_i(\cdot)\mathcal{I})$ (see Proposition \ref{belong-prop}). Observe also that if $\lambda_i(\nu_i,a_i)$ exists ($1\leq i\leq 3$),
then
$$
\lambda_i(\nu_i,a_i)=\tilde\lambda_i(\nu_i,a_i).
$$

Consider \eqref{competition-system-eq}. By general semigroup theory, for any $(u_0,v_0)\in X_1\times X_2$, \eqref{competition-system-eq} has a
unique (local) solution $(u(t, x; u_0, v_0),v(t, x; u_0, v_0))$ with $(u(0, x; u_0, v_0),v(0, x; u_0, v_0))=(u_0(x), v_0(x))$.

The main results of the current paper  are stated in the following four theorems.

\begin{theorem}[Effects of spatial variation]
\label{spatial-effect-thm}

\begin{itemize}
\item[(1)]  {\rm (Existence of principal eigenvalues)} For given $1\leq i\leq 2$, $\lambda_i(\nu_i,a_i)$ exists if
$a_{i,\max}-a_{i,\min}<\nu_i \inf_{x\in\bar D}\int_D k(y-x)dy$.

\item[(2)] {\rm (Existence of principal eigenvalues)}
For given $1\leq i\leq 2$, $\lambda_i(\nu_i,a_i)$ exists
if $h_i(\cdot)$ is in $C^N(\bar D)$, there is some $x_0\in {\rm
Int}(D)$ satisfying that $h_i(x_0)=h_{i,\max}$, and  the partial
derivatives of $h_i(x)$ up to order $N-1$ at $x_0$ are zero.

\item[(3)] {\rm (Upper bounds)} For given $1\le i\le 3$ and $c_i\in\RR$, $\sup\{\tilde \lambda_i(\nu_i,a_i)\,|\, a_i\in X_i,\,\, \hat a_i=c_i\}=\infty$.

\item[(4)] {\rm (Lower bounds)}  Assume that $k(\cdot)$ is symmetric with respect to $0$ (i.e. $k(-z)=k(z)$) and $i=2$. For given  $c_i\in\RR$,
$$\inf\{\tilde \lambda_i(\nu_i,a_i)\,|\, a_i\in X_i,\,\, \hat a_i=c_i\}=\lambda_i(\nu_i,c_i)(=c_i)$$
 (hence
$\tilde\lambda_i(\nu_i,a_i)\ge\tilde\lambda_i(\nu_i,\hat a_i)$). If the principal eigenvalue  of $\nu_i\mathcal{K}_i+h_i(\cdot)\mathcal{I}$ exists,
then "$=$" holds if and only if $a_i(\cdot)$ is a constant function, that is $a_i(\cdot)\equiv \hat a_i$.

\item[(5)] {\rm (Monotonicity)} For given $a_i^1,a_i^2\in X_i$, if $a_i^1(x)\le a_i^2(x)$, then $\tilde\lambda_i(a_i^1,\nu_i)\le \tilde\lambda_i(a_i^2,\nu_i)$
$(i=1,2,3)$.

\end{itemize}
\end{theorem}

\begin{remark}
\label{spatial-effect-rk}
\begin{itemize}
\item[(1)]  For the case $i=3$, similar result to Theorem \ref{spatial-effect-thm}(1)  is proved in \cite{ShZh0}.
To be more precise, it is proved in \cite{ShZh0} that if $a_{3,\max}-a_{3,\min}<\nu_3$, then $\lambda_3(\nu_3,a_3)$ exists.

\item[(2)]  For the case $i=3$, similar result to Theorem \ref{spatial-effect-thm}(2)  is  also proved in \cite{ShZh0}.
Actually  it is proved in \cite{ShZh0} that if $a_3(\cdot)$ is $C^N$ and there is $x_0\in\RR^N$ such that $a_3(x_0)=a_{3,\max}$ and the
partial derivatives of $a_3(x)$ up to order $N-1$ at $x_0$ are zero, then $\lambda_3(\nu_3,a_3)$ exists.

\item[(3)]
For one space dimensional random dispersal operators,  for given $c_i\in\RR$, $\sup\{\lambda_{R,i}(\nu_i,a_i)\,|\, a_i\in X_i^{++},\,\, \hat a_i=c_i\}<\infty$
(see Remark \ref{rk-1} for detail).
  Theorem \ref{spatial-effect-thm}(3)
hence reflects some difference between random dispersal operators and nonlocal dispersal operators.

\item[(4)] Similar result to Theorem \ref{spatial-effect-thm}(4) holds for $i=3$. To be more precise,
it is proved in   \cite{ShZh2} that for any given $c_3\in\RR$,
$$\inf\{\tilde \lambda_3(\nu_3,a_3)\,|\, a_3\in X_3,\,\, \hat a_3=c_3\}=\lambda_3(\nu_3,c_3)(=c_3).
$$ But
 Theorem \ref{spatial-effect-thm}(4) may not hold
  for the case $i=1$ (see Remark \ref{rk-1} for detail).
  \end{itemize}
\end{remark}

\begin{theorem}[Effects of dispersal rate]
\label{dispersal-rate-effect-thm}
 Assume that $1\le i\le 3$ and $k(\cdot)$ is symmetric with respect to $0$.
\begin{itemize}
\item[(1)] {\rm (Monotonicity)} Assume $a_i(\cdot)\not\equiv {\rm constant}$.   If $\nu_i^1<\nu_i^2$, then
$\tilde \lambda_i(\nu_i^1,a_i)>\tilde \lambda_i(\nu_i^2,a_i)$.

\item[(2)] {\rm (Existence of principal eigenvalue)}   If $i=1$ or $3$ and  $\lambda_i(\nu_i,a_i)$ exists for some $\nu_i>0$, then
$\lambda_i(\tilde\nu_i,a_i)$ exists for all $\tilde\nu_i>\nu_i$.

\item[(3)] {\rm (Existence of principal eigenvalue)}  There is $\nu_i^0>0$ such that  the principal eigenvalue $\lambda_i(\nu_i,a_i)$ of $\nu_i\mathcal{K}_i+h_i(\cdot)\mathcal{I}$
exists for $\nu_i>\nu_i^0$.

\item[(4)] {\rm (Limits as the dispersal rate goes to $0$)}
 $\lim_{\nu_i\to 0+}\tilde\lambda_i(\nu_i,a_i)=a_{i,\max}$.

 \item[(5)] {\rm (Limits as the dispersal rate goes to  $\infty$)}
 $\lim_{\nu_i\to\infty}\tilde\lambda_i(\nu_i,a_i)=-\infty$ for $i=1$ and $\lim_{\nu_i\to\infty}\tilde\lambda_i(\nu_i,a_i)=\hat a_i$ for $i=2$ and $3$.
\end{itemize}
\end{theorem}

\begin{remark}
\label{dispersal-rate-effect-rk}
\begin{itemize}
\item[(1)] It is open whether Theorem \ref{dispersal-rate-effect-thm} (2) holds for the case $i=2$.

\item[(2)] Theorem \ref{dispersal-rate-effect-thm} (3) and (4) still  hold if $k(\cdot)$ is not symmetric.
\end{itemize}
\end{remark}
For given $\delta>0$ and $\tilde k(\cdot):\RR^N\to \RR^+$ satisfying that ${\rm
supp}(\tilde k)=B(0,1):=\{z\in\RR^N\,|\, \|z\|<1\}$ and
$\int_{\RR^N}\tilde k(z)dz=1$, let
\begin{equation}
\label{k-delta}
k_\delta(z)=\frac{1}{\delta^N} \tilde k\left(\frac{z}{\delta}\right).
\end{equation}
When $k(z)=k_\delta(z)$, to indicate the dependence of $\tilde\lambda_i(\nu_i,a_i)$ on $\delta$,
put
$$
\tilde\lambda_i(\nu_i,a_i,\delta)=\tilde\lambda_i(\nu_i,a_i).
$$

\begin{theorem}[Effects of dispersal distance]
\label{dispersal-distance-effect-thm}
 Suppose that $k(z)=k_\delta(z)$, where  $k_\delta(z)$ is defined as in \eqref{k-delta} and $\tilde k(z)=\tilde k(-z)$.
Let $1\le i\le 3$.

\begin{itemize}

\item[(1)] {\rm (Limits as dispersal distance goes to $0$)} $\lim_{\delta\to 0}\tilde\lambda_i(\nu_i,a_i,\delta)=a_{i,\max}$.

\item[(2)] {\rm (Limits as dispersal distance goes to $\infty$)}
$\lim_{\delta\to\infty}\tilde \lambda_1(\nu_1,a_1,\delta)=-\nu_1+a_{1,\max}$,
$\lim_{\delta\to\infty}\tilde\lambda_2(\nu_2,a_2,\delta)=a_{2,\max}$, and
$\lim_{\delta\to\infty}\tilde\lambda_3(\nu_3,a_3,\delta)=\bar\lambda_3(\nu_3,a_3)$, where
$$
\bar\lambda_3(\nu_3,a_3)=\max\{{\rm Re}\lambda\,|\, \lambda\in \sigma( \nu_3 \mathcal{\bar I}+h_3(\cdot)\mathcal{I})\},
$$
and
$$
\mathcal{\bar I}u=\frac{1}{|D|}\int_D u(x)dx.
$$

\item[(3)] {\rm (Existence of principal eigenvalue)}  There is $\delta_0>0$ such that the principal eigenvalue $\lambda_i(\nu_i,a_i)$ of $\nu_i\mathcal{K}_i+h_i(\cdot)\mathcal{I}$
exists for $0<\delta<\delta_0$.
\end{itemize}
\end{theorem}

\begin{remark}
\label{dispersal-distance-effect-rk}
\begin{itemize}
\item[(1)] For $i=1$ or $3$,  Theorem \ref{dispersal-distance-effect-thm} (1) is actually  proved in \cite[Theorem 2.6]{KaLoSh1}.

\item[(2)] For $i=1$ or $3$, Theorem \ref{dispersal-distance-effect-thm} (3) is  proved in  \cite{KaLoSh1} (see also \cite{ShZh0} for the case $i=3$).

\end{itemize}
\end{remark}

\begin{corollary}[Criteria for the existence of principal eigenvalues]
\label{critria-thm}
Let $1\le i\le 3$ be given.
\begin{itemize}
\item[(1)] $\lambda_i(\nu_i,a_i)$ exists provided that $\max_{x\in\bar
D}a_i(x)-\min_{x\in\bar D} a_i(x)<\nu_i\inf_{x\in\bar D}\int_D k(y-x)dy$ in the case $i=1,2$ and $\max_{x\in\bar
D}a_i(x)-\min_{x\in\bar D} a_i(x)<\nu_i$ in the case $i=3$.

\item[(2)]  $\lambda_i(\nu_i,a_i)$ exists provided that $h_i(\cdot)$ is in $C^N(\bar D)$, there is some $x_0\in {\rm
Int}(D)$ satisfying that $h_i(x_0)=h_{i,\max}$, and  the partial
derivatives of $h_i(x)$ up to order $N-1$ at $x_0$ are zero.

\item[(3)]  There is $\nu_i^0>0$ such that  the principal eigenvalue $\lambda_i(\nu_i,a_i)$ of $\nu_i\mathcal{K}_i+h_i(\cdot)\mathcal{I}$
exists for $\nu_i>\nu_i^0$.

\item[(4)] Suppose that $k(z)=k_\delta(z)$, where
$k_{\delta}(z)$ is defined as in \eqref{k-delta} and $\tilde k(\cdot)$ is symmetric with respect to $0$. Then there is $\delta_0>0$ such that the principal eigenvalue $\lambda_i(\nu_i,a_i)$ of $\nu_i\mathcal{K}_i+h_i(\cdot)\mathcal{I}$
exists for $0<\delta<\delta_0$.
\end{itemize}
\end{corollary}

\begin{proof}
(1) and (2) are Theorem \ref{spatial-effect-thm}(1) and (2), respectively.

(3) is Theorem \ref{dispersal-rate-effect-thm}(3).

(4) is Theorem \ref{dispersal-distance-effect-thm}(3).

\end{proof}

\begin{theorem}
\label{competition-system-thm}
\begin{itemize}
\item[(1)]  There are $u^*(\cdot)\in X_1^{++}$ and $v^*(\cdot)\in X_2^{++}$ such that
$(u^*(\cdot),0)$ and $(0,v^*(\cdot))$ are stationary solutions of \eqref{competition-system-eq}. Moreover,
for any $(u_0,v_0)\in X_1^+\times X_2^+$ with $u_0\not = 0$ and $v_0=0$ (resp. $u_0=0$ and $v_0\not = 0$),
$(u(t,\cdot;u_0,v_0),v(t,\cdot;u_0,v_0))\to (u^*(\cdot),0)$ (resp.
$(u(t,\cdot;u_0,v_0),v(t,\cdot;u_0,v_0))\to (0,v^*(\cdot))$) as $t\to\infty$.

\item[(2)] For any $(u_0,v_0)\in (X_1^+\setminus\{0\})\times (X_2^+\setminus\{0\})$,
$\lim_{t\to \infty} (u(t,\cdot;u_0,v_0),v(t,\cdot;u_0,v_0))=(0,v^*(\cdot))$.
\end{itemize}
\end{theorem}

\section{Preliminary}

In this section, we present some preliminary materials to be used in the proofs of the main results in
later sections.

\subsection{Basic properties of solutions of nonlocal dispersal evolution equations}

In this subsection, we first present some basic properties of the solutions to the following  evolution equations associated to the
eigenvalue problems \eqref{dirichlet-op}, \eqref{neumann-op}, and \eqref{periodic-op},
\begin{equation}
\label{dirichlet-eq}
\p_t u(t,x)=\nu_1\left[\int_D k(y-x)u(t,y)dy-u(t,x)\right]+a_1(x)u(t,x),\quad\,\ \, x\in  \bar D,
\end{equation}
\begin{equation}
\label{neumann-eq}
\p_t u(t,x)= \nu_2\int_D k(y-x)[u(t,y)-u(t,x)]dy+a_2(x)u(t,x),\quad \quad x\in \bar D,
\end{equation}
and
\begin{equation}
\label{periodic-eq}
\begin{cases}
\p_t u(t,x)=\nu_3[\int_{\RR^N} k(y-x)u(t,y)dy-u(t,x)]+a_3(x)u(t,x),\quad
&x\in \RR^N, \cr u(t,x+p_j{\bf e_j})=u(t,x),\quad & x\in \RR^N,
\end{cases}
\end{equation}
respectively.

By general semigroup theory,  for any given $u_0\in X_1$ (resp. $u_0\in X_2$,
$u_0\in X_3$), \eqref{dirichlet-eq} (resp. \eqref{neumann-eq}, \eqref{periodic-eq}) has
a unique solution $u_1(t,\cdot;u_0,\nu_1,a_1)\in X_1$ (resp. $u_2(t,\cdot;u_0,\nu_2,a_2)\in X_2$, $u_3(t,\cdot;u_0,\nu_3,a_3)\in X_3$)
with $u_i(0,x;u_0,\nu_i,a_i)=u_0(x)$ ($i=1,2,3$).
As mentioned before, by general semigroup theory, for any given $(u_0,v_0)\in X_1\times X_2$, \eqref{competition-system-eq} also
has a unique (local) solution $(u(t,\cdot;u_0,v_0),v(t,\cdot;u_0,v_0))$ with $(u(0,x;u_0,v_0),v(0,x;u_0,v_0))=(u_0(x),v_0(x))$.

For given $u^1,u^2\in X_i$, we define
$$
u^1\le u^2,\quad {\rm if}\quad u^2-u^1\in X_i^+,
$$
and
$$
u^1\ll u^2,\quad {\rm if}\quad u^2-u^1\in X_i^{++}.
$$

\begin{definition}
A continuous function $u(t,x)$ on $[0,\tau)\times \bar D$ is called a
{\rm super-solution} (or {\rm sub-solution}) of
\eqref{dirichlet-eq} if for any $x \in \bar D$, u(t, x) is
differentiable  on $[0,\tau)$ and satisfies that
\begin{equation*} \partial _t u(t, x)\geq(or \leq)
\nu_1\Big[\int_{D}k(y-x)u(t,y)dy-u(t,x)\Big]+a_1(x)u(t,x)
\end{equation*} for   $t\in [0,\tau)$.
\end{definition}

{\it Super-solutions and sub-solutions} of \eqref{neumann-eq} and \eqref{periodic-eq} are defined in an analogous way.

\vspace{-.1in}\begin{proposition}[Comparison principle]
\label{comparison-prop} $\quad$
\begin{itemize}
\item[(1)]
If $u^1(t,x)$ and $u^2(t,x)$ are bounded sub- and super-solution of \eqref{dirichlet-eq}
(resp. \eqref{neumann-eq}, \eqref{periodic-eq})
on $[0,\tau)$,
respectively, and  $u^1(0,\cdot)\leq u^2(0,\cdot)$,  then
$u^1(t,\cdot)\leq u^2(t,\cdot)$ for $t\in [0,\tau).$

\item[(2)] For given $1\le i\le 3$, if $u^1,u^2\in X_i$, $u^1\leq u^2$ and $u^1\not\equiv u^2$, then $u_i(t,\cdot;u^1,\nu_i,a_i)\ll
u_i(t,\cdot;u^2,\nu_i,a_i)$ for all $t>0$.

\item[(3)] For given $1\le i\le 3$, $u_0\in X_i^+$, and $a_i^1, a_i^2\in X_i$, if
$a_i^1\le a_i^2$, then
$u_i(t,\cdot; u_0, \nu_i,a_i^1)\le u_i(t,\cdot;u_0, \nu_i,a_i^2)$ for $t\ge 0.$
\end{itemize}
\end{proposition}

\begin{proof}
(1) It follows from the arguments in \cite[Proposition 2.1]{ShZh0}.

(2) It follows from the arguments in \cite[Proposition 2.2]{ShZh0}.

(3) We consider the case $i=1$. Other cases can be proved similarly.

 Note that $u_1(t,x;\nu_1,a_1^2)$ is a supersolution of \eqref{dirichlet-eq} with $a_1(\cdot)$ being replaced by $a_1^1(\cdot)$.
Then by (1),
$$
u_1(t,\cdot;u_0, \nu_1,a_1^1)\leq u_1(t,\cdot;u_0, \nu_1,a_1^2)\quad \forall \,\, t\ge 0.
$$
\end{proof}

Next, we consider \eqref{competition-system-eq} and present some basic properties for solutions of the two species competition system.

 For given $(u^1, v^1), (u^2, v^2)\in X_1\times X_2$, we define
 $$
 (u^1, v^1)\le_1 (u^2, v^2),\quad {\rm if}\quad u^1(x)\le u^2(x),\,\, v^1(x)\le v^2(x),
 $$
 and
 $$
 (u^1, v^1)\le_2(u^2, v^2),\quad {\rm if}\quad u^1(x)\le u^2(x),\,\, v^1(x)\ge v^2(x).
 $$

Let $T>0$ and $(u(t, \!x), v(t, \!x))\in C([0,T)\times\bar{D},\ensuremath{\mathbb{R}}^2)$ with $(u(t,\cdot),v(t,\cdot))\in X_1^+\times X_2^+$. Then
$(u(t, \!x), v(t,\!x))$ is called a {\it super-solution} ({\it sub-solution}) of \eqref{competition-system-eq} on $[0,T)$ if
$$
\begin{cases}
\partial_tu(t, x)\geq (\leq)\nu [\int_D k(y-x)u(t,y)dy-u(t,x)]+u (t,x)f(x,u(t,x)+v(t,x)),\quad x\in\bar D,\cr
\partial_tv(t, x)\leq (\geq
)\nu\int_D k(y-x)[v(t,y)-v(t,x)]dy+v(t,x)f(x,u(t,x)+v(t,x)),\quad x\in \bar D,
\end{cases}
$$
for $t\in [0,T)$.

\begin{proposition}
\label{comparison-system-prop}
\begin{itemize}
\item[\rm(1)]\!\! If $(0,\, \!0)\!\leq_1\!(u_0,\!v_0)$, then $(0,0)\!\leq_1\!
(u(t,\cdot;u_0,\!v_0),v(t,\!\cdot;u_0,\!v_0))$ for all $t>0$
at which $(u(t,\cdot;u_0,v_0),v(t,\cdot;u_0,v_0))$ exists.

\item[\rm(2)] \!\! If $(0,\, \!0)\!\!\leq_1\!\!(u_i,\!v_i)$, for $i=1,2$,
$(u_1(0,\cdot),v_1(0,\cdot))\leq_2 (u_2(0,\cdot)$, $v_2(0,\cdot))$, and
$(u_1(t, \!x),v_1(t,\!x))$  and $(u_2(t, x),v_2(t,\!x))$  are a sub-solution and  a super-solution of \eqref{competition-system-eq} on
$[0,T)$ respectively, then $(u_1(t,\cdot),v_1(t,\cdot))\leq _2 (u_2(t,\cdot)$, $v_2(t,\cdot))$ for $t\in (0,T)$.

\item[\rm(3)] \!\! If $(0,\, \!0)\!\leq_1\!(u_i,\!v_i)$, for $i=1,2$ and $(u_1,v_1)\leq_2 (u_2,v_2)$, then
$$(u(t,\cdot;u_1,v_1), v(t,\cdot;u_1,v_1))\leq _2
(u(t,\cdot;u_2,v_2),v(t,\cdot;u_2,v_2))$$ for all $t\!>\!0$ at which both $(u(t,\cdot;u_1,v_1)$,
$v(t,\cdot;u_1,v_1))$ and $(u(t,\cdot;u_2,v_2)$, $v(t,\cdot;u_2,v_2))$ exist.

\item[\rm(4)] Let $(u_0,v_0)\in X_1^+\times X_2^+$, then $(u(t,\cdot;u_0,v_0),v(t,\cdot;u_0,v_0))$
exists for all $t>0$.
\end{itemize}
\end{proposition}

\begin{proof}
It follows from the arguments in Proposition 3.1 in \cite{HeNgSh}.
\end{proof}

\subsection{Basic properties of principal spectrum points of nonlocal dispersal operators}

In this subsection, we prove some basic properties of principal spectrum points/principal eigenvalues
of nonlocal dispersal operators.

First of all,  we derive some properties of  the  principal spectrum points of nonlocal dispersal operators
by using
the spectral radius of the solution operators of the associated evolution equations.
To this end, define $\Phi_i(t;\nu_i,a_i):X_i\to X_i$  by
\begin{equation}
\label{solution-op}
\Phi_i(t;\nu_i,a_i)u_0=u_i(t,\cdot;u_0,\nu_i,a_i),\quad u_0\in X_i,\,\,\, i=1,2,3.
\end{equation}
Let $r(\Phi_i(t;\nu_i,a_i))$ be the spectral radius of $\Phi_i(t;\nu_i,a_i)$. We have the following propositions.

\begin{proposition}
\label{belong-prop}
Let $1\le i\le 3$ be given.
\begin{itemize}
\item[(1)]
For given  $t>0$, $
e^{\tilde\lambda_i(\nu_i,a_i) t}=r(\Phi_i(t;\nu_i,a_i)).
$

\item[(2)] $\tilde\lambda_i(\nu_i,a_i)\in\sigma(\nu_i\mathcal{K}_i+h_i(\cdot)\mathcal{I})$.
\end{itemize}
\end{proposition}

\begin{proof}
 Observe that $\nu_i\mathcal{K}_i+h_i(\cdot)\mathcal{I}:X_i\to X_i$ is a bounded linear operator. Then by spectral mapping theorem,
\begin{equation}
\label{radius-eq1}
e^{\sigma (\nu_i\mathcal{K}_i+h_i(\cdot)\mathcal{I}) t}=\sigma(\Phi_i(t;\nu_i,a_i))\setminus\{0\}\quad \forall\,\, t>0.
\end{equation}
By Proposition \ref{comparison-prop},
\begin{equation}
\label{radius-eq2}
\Phi_i(t;\nu_i,a_i)X_i^+\subset X_i^+\quad \forall \,\, t>0.
\end{equation}
Hence  $\Phi_i(t;\nu_i,a_i)$ is a positive operator on $X_i$. Then by \cite[Proposition 4.1.1]{Mey},
 $r(\Phi_i(t;\nu_i,a_i)\in \sigma(\Phi_i(t;\nu_i,a_i))$ for any $t>0$. By \eqref{radius-eq1},
$$
e^{\tilde\lambda_i(\nu_i,a_i) t}=r(\Phi_i(t;\nu_i,a_i))\quad \forall\,\, t>0,
$$
and hence
$\tilde\lambda_i(\nu_i,a_i)\in \sigma(\nu_i\mathcal{K}_i+h_i(\cdot)\mathcal{I})$.
\end{proof}

\begin{proposition}
\label{most-basic-prop}
\begin{itemize}
\item[(1)] $\tilde\lambda_1(\nu_1,0)<0$.

\item[(2)] $\tilde\lambda_2(\nu_2,0)=0$.

\item[(3)] $\tilde\lambda_3(\nu_3,0)=0$.

\end{itemize}
\end{proposition}

\begin{proof}
(1) Let $u_0(x)\equiv 1$. Observe that
$$
\int_D k(y-x)u_0(y)dy-u_0(x)\leq 0,
$$
and there is $x_0\in D$ such that
$$
\int_D k(y-x_0)u_0(y)dy-u_0(x)<0.
$$
By Proposition \ref{comparison-prop}(2),
$$
0\ll \Phi_1(t; \nu_1, 0)u_0\ll   u_0\quad\forall \,\, t>0,
$$
and then
$$
\|\Phi_1(t;\nu_1,0)u_0\|<1\quad \forall \,\, t>0.
$$
Note that for any $\tilde u_0\in X_1$ with $\|\tilde u_0\|\le 1$, by Proposition \ref{comparison-prop}(2)
again,
$$
\|\Phi_1(t;\nu_1,0)\tilde u_0\|\le \|\Phi_1(t;\nu_1,0)u_0\|<1\quad \forall \,\, t>0.
$$
This implies that
$$r(\Phi_1(t;\nu_1,0))<1\quad\forall\,\, t>0,
$$
and then
$
\tilde\lambda_1(\nu_1,0)<0.$

(2) Let $u_0(\cdot)\equiv 1$. Observe that
$$
\Phi_2(t;\nu_2,0)u_0=u_0\quad \forall\,\, t\ge 0,
$$
and
$$
\|\Phi_2(t;\nu_2,0)\tilde u_0\|\le \|\Phi_2(t;\nu_2,0)u_0\|=1
$$
for all $t\ge 0$ and $\tilde u_0\in X_2$ with $\|\tilde u_0\|\le 1$.
It then follows that
$$
r(\Phi_2(t;\nu_2,0))=1\quad \forall\,\, t\ge 0,
$$
and then
$
\tilde\lambda_2(\nu_2,0)=0.$

(3) It can be proved by the similar arguments as in (2).

\end{proof}

Next, we prove some properties of principal spectrum points of nonlocal dispersal operators
by using the spectral radius of the induced nonlocal operators $U^i_{a_i,\nu_i,\alpha_i}$ and
$V^i_{a_i,\nu_i,\alpha_i}$ ($i=1,2,3$), where $\alpha_i>\max_{x\in D}h_i(x)$ ($i=1,2,3$),
\begin{equation}
\label{u-1-2-op}
(U^i_{a_i, \nu_i, \alpha_i}u)(x)={\int}_D\frac{\nu_i k(y-x)u(y)}{\alpha_i-h_i(y)}dy,\quad i=1,2,
 \end{equation}
 \begin{equation}
\label{u-3-op}
(U^3_{a_3,\nu_3,
\alpha_3}u)(x)=\int_{\RR^N}\frac{\nu_3k(y-x)u(y)}{\alpha_3-h_3(y)}dy,
\end{equation}
and
 \begin{equation}
 \label{v-1-2-op}
  (V^i_{ a_i, \nu_i, \alpha_i}u)(x)=\frac{\nu_i\mathop{\int}_{D}k(y-x)u(y)dy}{\alpha_i-h_i(x)}=\frac{\nu_i(\mathcal K_i u)(x)}{\alpha_i-h_i(x)},\quad i=1,2,
\end{equation}
\begin{equation}
\label{v-3-op}
 (V^3_{ a_3, \nu_3, \alpha_3}u)(x)=\frac{\nu_3\mathop{\int}_{\RR^N}k(y-x)u(y)dy}{\alpha_3-h_3(x)}=\frac{\nu_3(\mathcal K_3 u)(x)}{\alpha_3-h_3(x)}.
\end{equation}

Observe that $U^i_{ a_i, \nu_i, \alpha_i}$  and $V^i_{a_i,\nu_i,\alpha_i}$  are  positive and compact operators on $X_i$ ($i=1,2,3$). Moreover,
there is $n\ge 1$ such that
$$
\big(U_{a_i,\nu_i,\alpha_i}^i\big)^n (X_i^+\setminus\{0\})\subset X_i^{++},\quad i=1,2,3,
$$
and
$$
\big(V_{a_i,\nu_i,\alpha_i}^i\big)^n (X_i^+\setminus\{0\})\subset X_i^{++},\quad i=1,2,3.
$$
Then by Krein-Rutman Theorem,
\begin{equation}
\label{radius-in-spectrum}
r(U^i_{a_i,\nu_i,\alpha_i})\in\sigma(U^i_{a_i,\nu_i,\alpha_i}),\,\, \, r(V^i_{a_i,\nu_i,\alpha_i})\in\sigma(V^i_{a_i,\nu_i,\alpha_i}),
\end{equation}
and
$r(U^i_{a_i,\nu_i,\alpha_i})$ and $r(V^i_{a_i,\nu_i,\alpha_i})$ are isolated algebraically simple eigenvalues of $U^i_{a_i,\nu_i,\alpha_i}$ and
$V^i_{a_i,\nu_i,\alpha_i}$ with positive eigenfunctions, respectively.

\begin{proposition}
\label{relation-prop}
\begin{itemize}
\item[(1)]  $\alpha_i>h_{i,\max}$ is an eigenvalue of $\nu_i\mathcal{K}_i+h_i(\cdot)\mathcal{I}$ with $\phi(x)$ being an eigenfunction iff $1$ is an eigenvalue
of $U^i_{a_i,\nu_i,\alpha_i}$ with $\psi(x)=(\alpha_i-h_i(x))\phi(x)$ being an eigenfunction.

\item[(2)] $\alpha_i>h_{i,\max}$ is an eigenvalue of $\nu_i\mathcal{K}_i+h_i(\cdot)\mathcal{I}$ with $\phi(x)$ being an eigenfunction iff $1$ is an eigenvalue
of $V^i_{a_i,\nu_i,\alpha_i}$ with $\phi(x)$ being an eigenfunction.
\end{itemize}
\end{proposition}

\begin{proof}
It follows directly from the definitions of $U^i_{a_i,\nu_i,\alpha_i}$ and $V^i_{a_i,\nu_i,\alpha_i}$.
\end{proof}

 \begin{proposition}
 \label{radius-continuous-prop}  Let $1\le i\le 3$ be given.
\begin{itemize}
\item[(a)]
$r(U^i_{a_i,\nu_i,\alpha_i})$ is continuous in $\alpha_i (>h_{i,\max})$,  strictly decreases as $\alpha_i$ increases,  and $r(U^i_{a_i,\nu_i,\alpha_i})\to 0$ as $\alpha_i\to\infty$.

\item[(b)]  $r(V^i_{a_i,\nu_i,\alpha_i})$ is continuous in $\alpha_i (>h_{i,\max})$,  strictly decreases as $\alpha_i$ increases,  and $r(V^i_{a_i,\nu_i,\alpha_i})\to 0$ as $\alpha_i\to\infty$.
\end{itemize}
\end{proposition}

\begin{proof}
We prove (a) in the case $i=1$. The other cases can be proved similarly.

First, note that $r(U_{a_1,\nu_1,\alpha_1}^1)$ is an isolated algebraically simple eigenvalue of $U_{a_1,\nu_1,\alpha_1}^1$.
It then follows from the perturbation theory of the spectrum of bounded operators that $r(U_{a_1,\nu_1,\alpha_1}^1)$ is continuous in $\alpha_1(>h_{1,\max})$.

Next, we prove that $r(U_{a_1,\nu_1,\alpha_1}^1)$  is strictly decreasing as $\alpha_1$ increases. To this end, fix any $\alpha_1>h_{1,\max}$.  Let
$\phi_1(\cdot)$ be a positive eigenfunction of $U_{a_1,\nu_1,\alpha_1}^1$ corresponding to the eigenvalue $r(U_{a_1,\nu_1,\alpha_1}^1)$.
Note that for any given $\tilde\alpha_1>\alpha_1$, there is $\delta_1>0$ such that
$$
\frac{\tilde\alpha_1-\alpha_1}{\alpha_1-h_1(x)}>\delta_1\quad \forall \,\, x\in\bar D.
$$
This implies that
\begin{align*}
\big(U_{a_1,\nu_1,\tilde\alpha_1}^1\phi_1\big)(x)&=\int_D \frac{\nu_1 k(y-x)\phi_1(y)}{\tilde\alpha_1-h_1(y)}dy\\
&=\int_D \frac{\nu_1 k(y-x)\phi_1(y)}{\alpha_1-h_1(y)}\cdot\frac{1}{1+\frac{\tilde\alpha_1-\alpha_1}{\alpha_1-h_1(y)}}dy\\
&\le \frac{1}{1+\delta_1}\int_D \frac{\nu_1k(y-x)\phi_1(y)}{\alpha_1-h_1(y)}dy\\
&=\frac{r(U_{a_1,\nu_1,\alpha_1}^1)}{1+\delta_1}\phi_1(x)\quad \forall\,\, x\in\bar D.
\end{align*}
It then follows that
$$
r(U_{a_1,\nu_1,\tilde\alpha_1}^1)\le \frac{r(U_{a_1,\nu_1,\alpha_1}^1)}{1+\delta_1}<r(U_{a_1,\nu_1,\alpha_1}^1),
$$
and hence $r(U_{a_1,\nu_1,\alpha_1}^1)$ is strictly decreasing as $\alpha_1$ increases.

Finally, we prove that $r(U^1_{a_1,\nu_1,\alpha_1})\to 0$ as $\alpha_1\to\infty$.
Note that for any $\epsilon>0$, there is $\alpha_1^*>0$ such that for $\alpha_1>\alpha_1^*$,
$$
\int_D \frac{\nu_1 k(y-x)}{\alpha_1-h_1(y)}dy<\epsilon \quad \forall\,\, x\in\bar D.
$$
This implies that
$$
\|U_{a_1,\nu_1,\alpha_1}^1\|<\epsilon\quad \forall\,\, \alpha_1>\alpha_1^*.
$$
Hence  $r(U^1_{a_1,\nu_1,\alpha_1})\to 0$ as $\alpha_1\to\infty$.
\end{proof}

\begin{proposition}
\label{sufficient-cond-prop2}
Let $1\le i\le 3$ be given.
\begin{itemize}
\item[(a)] If  there is $\alpha_i>h_{i,\max}$ such that $r(U^i_{a_i,\nu_i,\alpha_i})>1$, then
$\tilde\lambda_i(\nu_i,a_i)>h_{i,\max}$.

\item[(b)]  If there is $\alpha_i>h_{i,\max}$ such that $r(V^i_{a_i,\nu_i,\alpha_i})>1$, then
$\tilde\lambda_i(\nu_i,a_i)>h_{i,\max}$.

\end{itemize}
\end{proposition}

\begin{proof}
We prove (b). (a) can be proved similarly.

Fix $1\le i\le 3$.
 Suppose  that there is $\alpha_i>h_{i,\max}$ such that $r(V^i_{a_i,\nu_i,\alpha_i})>1$. Then
By Proposition \ref{radius-continuous-prop},
there is $\alpha_0>h_{i,\max}$ such that
\begin{equation}
\label{radius-equal-1-eq2}
r(V^i_{a_i,\nu_i,\alpha_0})=1.
\end{equation}
By Proposition \ref{relation-prop}, $\alpha_0\in\sigma(\nu_i\mathcal{K}_i+h_i(\cdot)\mathcal{I})$. This implies that
$\tilde\lambda_i(\nu_i,a_i)\ge \alpha_0>h_{i,\max}.
$
\end{proof}

\begin{proposition}[Necessary and sufficient condition]
\label{iff-prop}
For given $1\leq i\leq 3$, $\lambda_i(\nu_i,a_i)$ exists if and only if
$\tilde\lambda_i(\nu_i,a_i)>h_{i, \max}$.
\end{proposition}

\begin{proof}
For $1\leq i \leq 3$, $\nu_i\mathcal K_i$  is a compact operator. Hence $\nu_i\mathcal K_i+h_i(\cdot)\mathcal I$ can be viewed as compact perturbation of the operator $h_i(\cdot)\mathcal I$. Clearly, the essential spectrum
$\sigma_{\rm ess}(h_i\mathcal{I})$ of $h_i(\cdot)\mathcal I$ is given by
\[
\sigma_{\mathrm{ess}}({h_i \mathcal I})=[h_{i,\min},h_{i,\max}].
\]
Since the essential spectrum is invariant under compact perturbations (see \cite{Spectral theory}), we have
\[
\sigma_{\mathrm{ess}}(\nu_i\mathcal K_i+h_i\mathcal I)=[h_{i,\min},h_{i,\max}],
\]
where $\sigma_{ess}(\nu_i \mathcal K_i+h_i\mathcal I)$  is the essential spectrum of $\nu_i\mathcal K_i+h_i(\cdot)\mathcal I$. Let
\[
\sigma_{disc}(\nu_i\mathcal K_i+h_i\mathcal I)=\sigma(\nu_i\mathcal K_i+h_i\mathcal I)\backslash\sigma_{ess}(\nu_i\mathcal K_i+h_i\mathcal I).
\]
Note that if $\lambda\in \sigma_{disc}(\nu_i\mathcal K_i+h_i\mathcal I)$, then it is an isolated eigenvalue of finite multiplicity.

On the one hand, if $\tilde \lambda_i(\nu_i, a_i)>h_{i, \max}(x)$,
then $\tilde \lambda_i(\nu_i,a_i) \in \sigma_{\text{disc}}(\nu_i \mathcal K_i+h_i \mathcal I)$.
By Proposition \ref{relation-prop}, $1\in \sigma\Big(U^i_{a_i,\nu_i,\tilde \lambda_i(\nu_i, a_i)}\Big)$.
Hence
$$
r\Big(U^i_{a_i,\nu_i,\tilde \lambda_i(\nu_i, a_i)}\Big)\ge 1.
$$
By Proposition \ref{radius-continuous-prop}, there is $\tilde{\tilde\lambda}\ge \tilde\lambda_i(\nu_i,a_i)$ such that
$$
r\Big(U^i_{a_i,\nu_i,\tilde {\tilde\lambda}}\Big)=1.
$$
This together with Proposition \ref{relation-prop} implies that $\tilde{\tilde\lambda}$ is
an isolated algebraically simple eigenvalue of $\nu_i \mathcal K_i+h_i (\cdot)\mathcal I$ with a positive eigenfunction.
 By Definition 2.1 (2), $\lambda_i(\nu_i, a_i)$  exists.

On the other hand, if $\lambda_i(\nu_i, a_i)$ exists, then $\tilde \lambda_i(\nu_i, a_i)= \lambda_i(\nu_i, a_i)\in \sigma_{\text{disc}}(\nu_i \mathcal K_i+h_i \mathcal I)$.
This implies that
 $\tilde\lambda_i(\nu_i,a_i)>h_{i, \max}(x)$.
\end{proof}

Finally, we present some variational characterization of the principal spectrum points of nonlocal dispersal operators when the kernel
function is symmetric.
In the rest of this subsection, we assume that $k(\cdot)$ is symmetric with respect to $0$.
Recall
\[
\mathcal K_3: X_3 \to X_3, \,\  (\mathcal K_3 u)(x)=\int_{\RR^N}k(y-x)u(y)dy \,\ \ \forall \,\, u\in X_3.
\]
For given $a\in X_3$, let
\begin{equation}
\label{periodic-kernal-eq}
\hat k(z)=\sum_{j_1,j_2,\cdots,j_N\in\ZZ} k(z+(j_1p_1,j_2p_2,\cdots, j_Np_N)),
\end{equation}
where $p_1, p_2, \cdots p_N$ are periods of $a(x)$.  Then $\hat k(\cdot)$ is also symmetric with respect to $0$ and
\begin{equation}
\label{K-P-equivalent-eq}
(\mathcal{K}_3u)(x)=\int_D \hat k(y-x)u(y)dy\quad \forall \,\, u\in X_3,
\end{equation}
where $D=[0, p_1]\times [0, p_2] \times \cdots \times [0, p_N]$ (see \eqref{periodic-domain-eq}).

\begin{proposition}
\label{variation-characterization-prop}
Assume that $k(\cdot)$ is symmetric with respect to $0$. Then
$$
\tilde\lambda_i(\nu_i,a_i)=\sup_{u\in L^2(D),\|u\|_{L^2(D)}=1}\int_D [\nu_i(\mathcal{K}_iu)(x)u(x)+h_i(x)u^2(x)]dx\quad (i=1, 2, 3).
$$
\end{proposition}

\begin{proof}
First of all, note that $\nu_i\mathcal{K}_i+h_i(\cdot)\mathcal{I}$ is also a bounded operator on $L^2(D)$ and $\nu_i\mathcal{K}_i$ is a compact operator on
$L^2(D)$, where $\mathcal{K}_i$ is defined as in \eqref{K-P-equivalent-eq} when $i=3$. Let $\sigma(\nu_i\mathcal{K}_i+h_i\mathcal{I},L^2(D))$ be the spectrum of
$\nu_i\mathcal{K}_i+h_i(\cdot)\mathcal{I}$ considered on $L^2(D)$ and
$$
\tilde\lambda(\nu_i,a_i,L^2(D))=\sup\{{\rm Re}\lambda\,|\, \lambda\in\sigma (\nu_i\mathcal{K}_i+h_i\mathcal{I},L^2(D))\}.
$$
Then we also have
$$
\tilde \lambda(\nu_i,a_i,L^2(D))\in\sigma (\nu_i\mathcal{K}_i+h_i\mathcal{I},L^2(D)),
$$
$$
[h_{i,\min},h_{i,\max}]\subset  \sigma (\nu_i\mathcal{K}_i+h_i\mathcal{I},L^2(D)),
$$
and
$$
\tilde \lambda(\nu_i,a_i,L^2(D))\ge h_{i,\max}.
$$
Moreover, if $\tilde\lambda_i(\nu_i,a_i)>h_{i,\max}$ (resp. $\tilde\lambda_i(\nu_i,a_i,L^2(D))>h_{i,\max}$), then
$ \tilde\lambda_i(\nu_i,a_i)$ (resp. $\tilde\lambda_i(\nu_i,a_i,L^2(D))$) is an eigenvalue of $\nu_i\mathcal{K}_i+h_i\mathcal{I}$ considered
on $L^2(D)$ (resp. $C(\bar D)$) and hence $\tilde\lambda_i(\nu_i,a_i,L^2(D))\ge\tilde\lambda_i(\nu_i,a_i)$ (resp. $\tilde\lambda_i(\nu_i,a_i)\ge \tilde\lambda_i(\nu_i,a_i,L^2(D))$).
We then must have
$$\tilde\lambda_i(\nu_i,a_i)=\tilde\lambda_i(\nu_i,a_i,L^2(D)).$$

Assume  now that  $k(\cdot)$ is symmetric with respect to $0$, that is, $k(-z)=k(z)$ for any $z\in\RR^N$. Then for any $u,  v\in L^2(D)$, in the case $i=1,2$,
\begin{align*}
\int_D (\mathcal{K}_i u)(x) v(x)dx&=\int_D\int_D k(y-x)u(y)v(x)dydx\\
&=\int_D\int_D k(x-y)u(x)v(y)dxdy\\
&=\int_D \int_D k(y-x)v(y)u(x)dydx\\
&=\int_D (\mathcal {K}_iv)(x)u(x)dx
\end{align*}
and in the case $i=3$,
\begin{align*}
\int_D (\mathcal{K}_3 u)(x) v(x)dx&=\int_D\int_D \hat k(y-x)u(y)v(x)dydx\\
&=\int_D\int_D\hat k(x-y)u(x)v(y)dxdy\\
&=\int_D \int_D \hat k(y-x)v(y)u(x)dydx\\
&=\int_D (\mathcal {K}_3v)(x)u(x)dx.
\end{align*}
Therefore $\mathcal{K}_i:L^2(D)\to L^2(D)$ is self-adjoint. By classical variational formula (see \cite{DoVa}), we have
$$
\tilde\lambda_i(\nu_i,a_i,L^2(D))=\sup_{u\in L^2(D),\|u\|_{L^2(D)}=1}\int_D [\nu_i(\mathcal{K}_iu)(x)u(x)+h_i(x)u^2(x)]dx.
$$
The proposition then follows.
\end{proof}

\subsection{A technical lemma}

In this subsection, we provide a useful technical lemma.

\begin{lemma}
\label{technical-lm}
Let $1\le i\le 3$ and $a_i\in X_i$ be given. For any $\epsilon>0$, there is $a_i^\epsilon\in X_i$ such that
$$
\|a_i-a_i^\epsilon\|<\epsilon,
$$
 $h_i^\epsilon(x)=-\nu_i+a_i^\epsilon(x)$ for $i=1$ or $3$ and
$h_i^\epsilon(x)=-\nu_i \int_D k(y-x)dy+a_i^\epsilon(x)$ for $i=2$ is in $C^N$, and   satisfies the following vanishing condition: there is $x_0\in {\rm Int}(D)$ such that
$h_i^\epsilon(x_0)=\max_{x\in\bar D}h_i^\epsilon(x)$ and the partial derivatives of $h_i^\epsilon(x)$ up to order $N-1$ at
$x_0$ are zero.
\end{lemma}

\begin{proof}
 We prove the case  $i=2$. Other cases can be proved similarly.

 First, let $\tilde x_0\in \bar D$ be such that
$$
h_2(\tilde x_0)=\max_{x\in\bar D} h_2(x).
$$
For any $\epsilon>0$, there is $\tilde x_\epsilon \in {\rm Int}(D)$  such that
\begin{equation}
\label{eq1}
h_2(\tilde x_0)-h_2(\tilde x_\epsilon)<\frac{\epsilon}{3}.
\end{equation}
Let $\tilde\sigma>0$ be such that
$$
B(\tilde x_\epsilon,\tilde\sigma)\Subset D,
$$
where $B(\tilde x_\epsilon,\tilde\sigma)$ denotes the open ball with center $\tilde x_\epsilon$ and radius
$\tilde\sigma$.

Note that there is $\xi(\cdot)\in C(\bar D)$  such that $0\leq \xi(x)\leq 1$, $ \xi(\tilde x_\epsilon)=1$, and
${\rm supp}( \xi)\subset B(\tilde x_\epsilon,\tilde\sigma)$.
\begin{equation}
\label{eq0}
 h_{2,\epsilon}(x)=h_2(x)+\frac{\epsilon}{3} \xi(x).
\end{equation}
Then  $ h_{2,\epsilon}(\cdot)$ is continuous on $D$ and  $h_{2,\epsilon}(\cdot)$ attains its maximum in ${\rm Int}(D)$.

Let $\tilde D\subset \RR^N$ be such that $D\Subset \tilde D$. Note that $h_{2,\epsilon}(\cdot)$ can be continuously extended to $\tilde D$.
Without loss of generality, we may then assume that $h_{2,\epsilon}(\cdot)$ is a continuous function on $\tilde D$ and there is
$x_0\in {\rm Int}(D)$  such that
$h_{2,\epsilon}(x_0)=\sup_{x\in\tilde D}h_{2,\epsilon}(x)$.
Observe that there is $\sigma>0$ and $\bar h_{2,\epsilon}(\cdot)\in C(\tilde D)$  such that
$B(x_0,\sigma)\Subset D$,
\begin{equation}
\label{eq3}
0\leq\bar h_{2,\epsilon}(x)-h_{2,\epsilon}(x)\le\frac{ \epsilon}{3}\quad \forall \,  x\in\tilde D,
\end{equation}
and
$$
\bar h_{2,\epsilon}(x)=h_{2,\epsilon}(x_0)\quad \forall \, x\in B(x_0,\sigma).
$$

Let
$$
\eta(x)=\begin{cases} C\exp(\frac{1}{\|x\|^2-1})\quad &{\rm if}\,\ \|x\|<1,\cr\cr
0\quad &{\rm if}\,\ \|x\|\geq 1,
\end{cases}
$$
where $C>0$ is such that $\int_{\RR^N}\eta(x)dx=1$.
For given $\delta>0$, set
$$
\eta_\delta(x)=\frac{1}{\delta^N}\eta(\frac{x}{\delta}).
$$
Let
$$
h_{2,\epsilon,\delta}(x)=\int_{\tilde D}\eta_\delta(y-x)\bar h_{2,\epsilon}(y)dy.
$$
By \cite[Theorem 6, Appendix C]{Eva}, $ h_{2,\epsilon,\delta}(\cdot)$ is in $C^\infty(\tilde D)$ and when $0<\delta\ll 1$,
\begin{equation}
\label{eq4}
|h_{2,\epsilon,\delta}(x)-\bar h_{2,\epsilon}(x)|<\frac{\epsilon}{3}\quad \forall \, x\in\bar D.
\end{equation}
It is not difficulty to see that for $0<\delta\ll 1$,
$$
h_{2,\epsilon,\delta}(x)\leq
\bar h_{2,\epsilon}(x_0)\quad \forall x\in B(x_0,\sigma),
$$
and
$$
h_{2,\epsilon,\delta}(x)=\bar h_{2,\epsilon}(x_0)\quad
\forall x\in B(x_0,\sigma/2).
$$

Fix $0<\delta\ll 1$. Let
$$h_{2}^{\epsilon}(x)=h_{2,\epsilon,\delta}(x).
$$
Then
$h_{2}^{\epsilon}(\cdot)$ attains its maximum at some $x_0\in{\rm Int}(D)$,
and the partial derivatives of $h_{2}^{\epsilon}(\cdot)$ up to order $N-1$ at $x_0$ are zero.
Let
$$
a_{2}^{\epsilon}(x)=h_2^\epsilon(x)+\nu_2\int_D k(y-x)dy\quad\forall \, x\in\bar D.
$$
Then $a_{2}^{\epsilon}\in{X}_2$,
$-\nu_2\int_D k(y-x)dy+a_2^\epsilon(x)=h_2^\epsilon(x)$,  and
$$
\|a_2-a_{2}^{\epsilon}\|=\|h_2^\epsilon-h_2\|\leq \|h_2^\epsilon- \bar h_{2,\epsilon}\|+\|\bar h_{2,\epsilon}-h_{2,\epsilon}\|+\|h_{2,\epsilon}-h_2\|<\epsilon.
$$
 The
lemma is thus proved.
\end{proof}

\section{Effects of Spatial Variations and the Proof of Theorem \ref{spatial-effect-thm}}

In this section, we investigate the effects of spatial variations on the principal spectrum points/principal eigenvalues
of nonlocal dispersal operators and prove Theorem \ref{spatial-effect-thm}.

First of all,
for given $1\le i\le 3$ and $c_i\in\RR$, let
$$
X_i(c_i)=\{ a_i\in X_i\,|\, \hat a_i=c_i\}
$$
(see \eqref{spacial-average} for the definition of $\hat a_i$). For given $x_0\in\RR^N$ and $\sigma>0$, let
$$
B(x_0, \sigma)=\{y\in\RR^N\,|\, \|y-x_0\|<\sigma\}.
$$

\begin{proof}[Proof of Theorem \ref{spatial-effect-thm}]

(1) We first prove the case $i=1$.  Let $x_0\in\bar D$ be such that
$$
h_1(x_0)=h_{1,\max}.
$$
Note that there is $\epsilon_0>0$ such that
$$
0\le a_1(x_0)-a_1(x)<\nu_1\inf_{x\in\bar D}\int_D k(y-x)dy-\epsilon_0\le \nu_1\int_D k(y-x)dy-\epsilon_0\quad \forall\,\,  x\in\bar D.
$$
For any $0<\epsilon<\epsilon_0$, put
$$
\lambda_\epsilon=h_1(x_0)+\epsilon (=-\nu_1+a_1(x_0)+\epsilon).
$$
Then
\begin{align*}
\frac{\nu_1\int_D k(y-x)dy}{\lambda_\epsilon-h_1(x)}&=\frac{\nu_1\int_D k(y-x)dy}{a_1(x_0)-a_1(x)+\epsilon}\\
&\ge \frac{\nu_1\int_D k(y-x)dy}{\nu_1\int_D k(y-x)dy+\epsilon-\epsilon_0}\\
&>1\quad \forall x\in\bar D.
\end{align*}
This implies
$$
r(V^1_{a_1,\nu_1,\lambda_\epsilon})>1\quad \forall\,\,  0<\epsilon\ll 1.
$$
Then by Proposition \ref{sufficient-cond-prop2} (b), $\tilde\lambda_1(\nu_1,a_1)>h_{1,\max}$. By
Proposition \ref{iff-prop}, $\lambda_1(\nu_1,a_1)$ exists.

We now prove the case $i=2$. Similarly,
let $x_0\in\bar D$ be such that
$$
h_2(x_0)=h_{2,\max}.
$$
Note that there is $\epsilon_0>0$ such that
$$
0\le a_2(x_0)-a_2(x)<\nu_2\inf_{x\in\bar D}\int_D k(y-x)dy-\epsilon_0\le \nu_2\int_D k(y-x_0)dy-\epsilon_0.
$$
For any $0<\epsilon<\epsilon_0$, put
$$
\lambda_\epsilon=h_2(x_0)+\epsilon(=-\nu_2\int_D k(y-x_0)dy+a_2(x_0)+\epsilon).
$$
Then
\begin{align*}
\frac{\nu_2\int_D k(y-x)dy}{\lambda_\epsilon-h_2(x)}&=\frac{\nu_2\int_D k(y-x)dy}{a_2(x_0)-\nu_2 \int_D k(y-x_0)dy+\nu_2\int_D k(y-x)dy-a_2(x)+\epsilon}\\
&\ge \frac{\nu_2\int_D k(y-x)dy}{\nu_2\int_D k(y-x)dy+\epsilon-\epsilon_0}\\
&>1\quad \forall x\in\bar D.
\end{align*}
This again implies that
$$
r(V^2_{a_2,\nu_2,\lambda_\epsilon})>1\quad \forall\,\,  0<\epsilon\ll 1.
$$
Then by Proposition \ref{sufficient-cond-prop2} (b), $\tilde\lambda_2(\nu_2,a_2)>h_{2,\max}$. By
Proposition \ref{iff-prop}, $\lambda_2(\nu_2,a_2)$ exists.

(2) It can be proved by the similar arguments as in \cite[Theorem B(2)]{ShZh0}.
For the completeness, we provide a proof in the following.

 Let $x_0\in {\rm Int}(D)$ be such that $h_i(x_0)=h_{i,\rm{max}}$ and  the partial derivatives of $h_i(x)$ up to order $N-1$ at $x_0$ are zero. Then  there is $M>0$ such that
$$
h_i(x_0)-h_i(y)\leq M||x_0-y||^N \quad  \forall \,\ y\in D.
$$
Fix  $\sigma>0$  such that  $B(x_0, 2\sigma)\subset D$ and $B(0, 2\sigma)\Subset{\rm supp}(k(\cdot))$. Let $v^*\in X_i^+$ be such that
\begin{equation*}
v^*(x)=
\begin{cases}
1 \quad &\forall\,\ x\in B(x_0, \sigma),\\
0\quad & \forall\,\ x\in D\backslash B(x_0, 2\sigma).
\end{cases}
\end{equation*}
Clearly, for every $x\in D\backslash B(x_0, 2\sigma)$ and $\gamma>1$, we have
\begin{equation}
\label{thm1-eq1}
(U^i_{a_i,\nu_i, h_{i}(x_0)+\epsilon}v^*)(x)\ge \gamma v^*(x)=0\quad \forall\,\, \epsilon>0.
\end{equation}
Note that there is $\tilde M>0$ such that for any $x\in B(x_0, 2\sigma)$,
$$
k(y-x)\geq \tilde M \quad \forall \,\ y\in B(x_0, \sigma).
$$
It then follows that for $x\in B(x_0, 2\sigma)$
\begin{align*}
(U^i_{a_i,\nu_i,h_{i}(x_0)+\epsilon}v^*)(x)&=\int_D\frac{\nu_ik(y-x)v^*(y)}{h_i(x_0)+\epsilon-h_i(y)}dy\\
&\geq \mathcal{\int}_{B(x_0, \sigma)}\frac{\nu_ik (y-x)}{M||x_0-y||^N+\epsilon}dy\\
&\geq \mathcal{\int}_{B(x_0, \sigma)}\frac{\nu_i \tilde M}{M||x_0-y||^N+\epsilon}dy.
\end{align*}
Notice that $\mathcal{\int}_{B(x_0, \sigma)}\frac{\tilde M}{M||x_0-y||^N}dy=\infty$. This implies that for $0<\epsilon \ll 1$, there is $\gamma>1$ such that
\begin{equation}
\label{thm1-eq2}
(U^i_{a_i,\nu_i,h_{i}(x_0)+\epsilon}v^*)(x)>\gamma v^*(x) \quad \forall \,\ x \in B(x_0, 2\sigma).
\end{equation}
By \eqref{thm1-eq1} and \eqref{thm1-eq2},
$$
U^i_{a_i,\nu_i,h_{i}(x_0)+\epsilon} v^*(x)\ge \gamma v^*(x)\quad \forall\,\ x\in D.
$$
Hence, $r(U^i_{a_i,\nu_i,h_{i}(x_0)+\epsilon})>1$. By Proposition \ref{sufficient-cond-prop2}(a),  $\tilde \lambda_i(\nu_i, a_i)>h_i(x_0)=h_{i, \max}$. By Proposition \ref{iff-prop}, the principle eigenvalue $\lambda_i(\nu_i, a_i)$ exists.

(3) Recall that $\tilde \lambda_i(\nu_i, \tilde a)=\sup\{\mathrm{Re}\mu|\mu\in \sigma(\nu_i\mathcal K_i+\tilde h_i(\cdot)\mathcal I)\}$
with $\tilde h_i(x)=-\nu_i+\tilde a(x)$ for $i=1, 3 $ and
  $\tilde h_i(x)=-\nu_2\int_D k(y-x) d y+\tilde a(x)$ for $i=2$.
 By the arguments of Proposition \ref{iff-prop},
  $$
  \sigma_{\mathrm{ess}}(\nu_i\mathcal K_i+\tilde h_i\mathcal I)=[\min_{x\in \bar D}\tilde h_i(x), \max_{x\in \bar D}\tilde h_i(x)].
  $$
  Note that
  $$\sup_{\tilde a\in X_i(c_i)}(\max_{x\in \bar D}\tilde a(x))=\infty.
  $$
 Then
\[
\sup_{\tilde a\in X_i(c_i)}\tilde \lambda_i(\nu_i, \tilde a)\geq \sup_{\tilde a\in X_i(c_i)}(\max_{x\in D}\tilde h_i(x))\geq-\nu_i+\sup_{\tilde a\in X_i(c_i)}(\max_{x\in D}\tilde a(x))=\infty.
\]

(4) We first assume that the principal eigenvalue $\lambda_2(\nu_2,a_2)$ exists.
 Suppose that $u_2(x)$ is a strictly positive principal eigenfunction with respect to the eigenvalue $\lambda_2(\nu_2, a_2)$. We divide both sides of (1.2) by $u_2(x)$ and integrate
 with respect to $x$ over $D$ to obtain
\[
\int_D\left[\frac{\nu_2[\int_Dk(y-x)(u_2(y)-u_2(x))dy]+a_2(x)u_2(x)}{u_2(x)}\right]dx=\int_D\lambda_2(\nu_2, a_2)dx,
\]
or
\begin{align*}
\lambda_2(\nu_2, a_2)&=\frac{\nu_2}{|D|}\int_D\int_Dk(y-x)\frac{u_2(y)-u_2(x)}{u_2(x)}dydx+\frac{1}{|D|}\int_Da_2(x)dx\\
&=\frac{\nu_2}{|D|}\int_D\int_Dk(y-x)\frac{u_2(y)-u_2(x)}{u_2(x)}dydx+\hat a_2.
\end{align*}
By the symmetry of $k(\cdot)$,
\begin{align}
\label{thm2-1-eq1}
&\int_D\int_Dk(y-x)\frac{u_2(y)-u_2(x)}{u_2(x)}dydx\nonumber\\
&=\frac{1}{2}\int\int_{D\times D} k(y-x)\frac{u_2(y)-u_2(x)}{u_2(x)}dydx+\frac{1}{2}\int\int_{D\times D}k(y-x)\frac{u_2(y)-u_2(x)}{u_2(x)}dydx\nonumber\\
&=\frac{1}{2}\int\int_{D\times D}k(y-x)\frac{u_2(y)-u_2(x)}{u_2(x)}dydx+\frac{1}{2}\int\int_{D\times D}k(y-x)\frac{u_2(x)-u_2(y)}{u_2(y)}dydx\nonumber\\
&=\frac{1}{2}\int\int_{D\times D}k(y-x)\frac{(u_2(y)-u_2(x))^2}{u_2(x)u_2(y)}dydx\nonumber\\
&\geq 0.
\end{align}
So,
\[
\inf\{\lambda_2(\nu_2, a_2)| a_2\in X_2, \hat a_2=c_2\}\geq \hat a_2=c_2.
\]
And clearly, $ \lambda_2(\nu_2, \hat a_2)=\hat a_2$. Together, we get
\[
\inf\{\lambda_2(\nu_2, a_2)| a_2\in X_2, \hat a_2=c_2\}= \lambda_2(\nu_2, \hat a_2)=c_2.
\]

Second, by Lemma 3.1, for   any $\epsilon>0$, there is $a_2^{\epsilon}\in X_2\cap C^N$, such that
\[
\|a_2-a_2^{\epsilon}\|<\epsilon,
\]
and $h_2^{\epsilon}(\cdot)\in C^N (=-\nu_2\int_D k(y-x)dy+a_2^{\epsilon})$  satisfies the vanishing condition in Theorem 2.1 (2). So, the principal eigenvalue $\lambda_2(\nu_2,a_2^{\epsilon})$
exists and
$\tilde \lambda_2(\nu_2, a_2^{\epsilon})=\lambda_2(\nu_2, a_2^{\epsilon})$.
By the above arguments,
\begin{equation}
\label{thm2-1-eq2}
\tilde \lambda_2(\nu_2, a_2^{\epsilon})=\lambda_2(\nu_2, a_2^{\epsilon})\geq  \lambda_2(\nu_2, \hat a_2^{\epsilon})=\hat a_2^{\epsilon}.
\end{equation}
We claim that
\[
\lim_{\epsilon \to 0}\tilde\lambda_2(\nu_2, a_2^{\epsilon})=\tilde \lambda_2(\nu_2, a_2).
\]
In fact, $\|a_2^{\epsilon}-a_2\|\leq \epsilon$, that is
\[
a_2(x)-\epsilon\le a_2^{\epsilon}(x)\leq a_2(x)+\epsilon \quad \forall \,\ x\in \bar D.
\]
 Note that $\Phi_2(t; \nu_2, a_2+\epsilon)u_0=e^{\epsilon t}\Phi_2(t;\nu_2, a_2)u_0$, where $\Phi_2(t;\nu_2, a_2)u_0$ is the solution of (3.2) with the initial value $u_0(\cdot)$. Similarly, we have $\Phi_2(t; \nu_2, a_2-\epsilon)u_0=e^{-\epsilon t}\Phi_2(t;\nu_2, a_2)u_0$. So
\[
r(\Phi_2(t;\nu_2, a_2\pm \epsilon))=e^{\pm \epsilon t}r(\Phi_2(t; \nu_2, a_2)).
\]
Hence
\begin{equation}
\label{thm2-1-eq3}
\tilde \lambda_2(\nu_2, a_2\pm \epsilon)=\tilde \lambda_2 (\nu_2, a_2)\pm \epsilon.
\end{equation}
By Proposition 3.1, we have
\[
\Phi_2(t;\nu_2, a_2-\epsilon)u_0\le \Phi_2(t;\nu_2, a_2^{\epsilon})u_0\leq \Phi_2(t;\nu_2, a_2+\epsilon)u_0.
\]
Hence
\[
r(\Phi_2(t;\nu_2, a_2-\epsilon))\le r(\Phi_2(t;\nu_2, a_2^{\epsilon}))\leq r(\Phi_2(t;\nu_2, a_2+\epsilon)).
\]
By\eqref{thm2-1-eq3},
\begin{equation*}
\tilde \lambda_2(\nu_2, a_2-\epsilon)\le \tilde \lambda_2(\nu_2, a_2^{\epsilon})\leq \tilde \lambda_2(\nu_2, a_2+\epsilon).
\end{equation*}
Taking the limit of \eqref{thm2-1-eq2} as $\epsilon \to 0$, we have
\[
\tilde \lambda_2 (\nu_2, a_2)\geq  \hat a_2
\]
So, $\rm{inf}\{\tilde \lambda_2(\nu_2, a_2)| a_2\in X_2, \hat a_2=c_2\}=\lambda_2(\nu_2, c_2)(=c_2)$.

When the principal eigenvalue exists, it is not difficult to prove that the $"="$ holds if and only if $a_2(\cdot)\equiv c_2$. In fact,
suppose that $\lambda_2(\nu_2,a_2)$ exists and $u_2(\cdot)$ is a corresponding positive eigenfunction.
By \eqref{thm2-1-eq1}, $\lambda_2(\nu_2,a_2)=\hat a_2(=c_2)$ iff $u_2(x)=u_2(y)$ for all $x,y\in \bar D$.
Hence $\lambda_2(\nu_2,a_2)=\hat a_2(=c_2)$ iff $u_2(\cdot)\equiv$constant, which implies that $a_2(x)=\lambda_2(\nu_2,a_2)=\hat a_2$.

(5) Suppose that $a_i^1,a_i^2\in X_i$ and $a_i^1\le a_i^2$. By Proposition \ref{comparison-prop}, for any $u_0\in X_i^+$ and $t\ge 0$,
$$
\Phi_i(t;\nu_i,a_i^1)u_0\le \Phi_i(t;\nu_i,a_i^2)u_0.
$$
This implies that
$$
r(\Phi_i(t;\nu_i,a_i^1))\le r(\Phi_i(t;\nu_i,a_i^2)).
$$
By Proposition \ref{belong-prop}, we have
$$
\tilde\lambda_i(\nu_i,a_i^1)\le \tilde\lambda_i(\nu_i,a_i^2).
$$
\end{proof}

\begin{remark}
\label{rk-1}
(1) Theorem 2.1 (3) is not true in the random dispersal case when the space dimension is one.
In fact,  for $1\leq i \leq 3$, we have $\lambda_{R, i}\leq c_i+{c_i}^2L^2$ for any $a_i(\cdot)\in X_i^{++}$, $\hat a_i=c_i$
and $D=(0, L)$. For the periodic boundary case, see Lemma 4.1 in \cite{LLM}. The proof of Neumann or Dirichlet
boundary case is similar to that of the  periodic boundary case.

{\rm
We give a proof for the Neumann boundary case. Let $\psi(x)$ be the eigenvalue function of the operator
$ \Delta +a_2(\cdot)\mathcal I$ defined on $C^2([0, L])$ with Neumann boundary condition. So $\psi(x)>0$ and we have
\begin{equation*}
\begin{cases}
\psi''(x)+a_2(x)\psi(x)=\lambda_{R, 2}\psi(x), \quad & x\in (0, L),  \\
\frac{\partial \psi}{\partial n}(x)=0, & x=0 \text{ or } L.
\end{cases}
\end{equation*}
Multiplying this by $\psi(x)$ and integrating it from $0$ to $L$, we have
\[
-\int_0^L\psi'^2(x)dx+\int_0^La_2(x)\psi^2(x)dx=\lambda_{R, 2}\int_0^L\psi^2(x)dx.
\]
Hence
\[
\lambda_{R, 2}=\frac{-\int_0^L\psi'^2(x)dx+\int_0^L a_2(x)\psi^2(x)dx}{\int_0^L\psi^2(x)dx}.
\]
Take $x_1, x_2\in[0, L)$, we have
\[
\psi^2(x_2)-\psi^2(x_1)=\int_{x_1}^{x_2}2\psi(x)\psi'(x)dx.
\]
Hence, for any positive number $k>0$,
\[
\psi^2(x_2)-\psi^2(x_1)\leq \frac 1k \int_0^L\psi'^2(x)dx+k\int_0^L\psi^2(x)dx.
\]
Multiplying the above inequality by $a_2(x_2)$ and integrating it with respect to $x_1\in [0, L )$ and $x_2\in [0, L)$, we get
\[
L\int_0^La_2(x_2)\psi^2(x_2)dx_2-c_2L\int_0^L\psi^2(x_1)dx_1\leq c_2L^2\left( \frac 1k \int_0^L\psi'^2(x)dx+k\int_0^L\psi^2(x)dx \right).
\]
This is equivalent to
\[
L\int_0^La_2(x)\psi^2(x)dx-c_2L\int_0^L\psi^2(x)dx\leq c_2L^2\left( \frac 1k \int_0^L\psi'^2(x)dx+k\int_0^L\psi^2(x)dx \right).
\]
Letting $k=c_2L$, we obtain
\[
-\int_0^L\psi'^2(x)dx+\int_0^La_2(x)\psi^2(x)dx\leq (c_2+c_2^2L^2)\int_0^L\psi^2(x)dx.
\]
So, we have
\[
\lambda_{R, 2}\leq c_2+c_2^2L^2.
\]

}

(2) Theorem 2.1 (4) may not be true for the Dirichlet type boundary condition. That is, $\tilde \lambda_1(\nu_1, a_1)\geq  \lambda_1(\nu_1, \hat a_1) $  may not be  true,
 where $a_1\in X_1 $.

In the random dispersal case, There is an  example in \cite{ShXi1} which  shows that the principal eigenvalue
$ \lambda_{R, 1}(\nu_1, a_1)$ of (1.4) is smaller than the principal eigenvalue $\lambda_{R, 1}(\nu_1, c_1)$ of (1.4) with $a_1(x)$ being replaced by
$c_1(=\hat a_1)$. It is prove in \cite{HeNgSh2} that
\[
\tilde\lambda_1(\nu_1, a_1, \delta)\to \lambda_{R, 1}(\nu_1, a_1)
\]
as $\delta \to 0$. So, for any $0<\delta\ll 1$, $\tilde\lambda _1(\nu_1, a_1, {\delta})$ is close to $\lambda_{R, 1}(\nu_1, a_1)$, and $\tilde \lambda_1(\nu_1, c_1, {\delta})$ is close to $\lambda_{R, 1}(\nu_1,c_1)$.
Hence $\tilde \lambda_1(\nu_1, a_1, {\delta})$ can be smaller than $\tilde \lambda_1(\nu_1, c_1, \delta)=\lambda_1(\nu_1, c_1, \delta)$ for $\delta\ll 1$.

(3) Theorem 2.1 (4) holds for periodic case (see  \cite{ShZh2}).
When $\lambda_i(\nu_i,a_i)$ does not exist ($i=2,3$), we may have $\tilde\lambda_i(\nu_i,a_i)=\hat a_i$, but $a_i(\cdot)$ is not a constant function.
For example,
let $X_3\!=\!\{u(x)\!\in \!C(\RR^N, \RR)| u(x+{\bf e_j})=u(x)), x\in \RR^N, j=1, 2, \cdots, N\}$,
and $q\in X_3$ with
\begin{equation*}
q(x)=
\begin{cases}
e^{\frac{\|x\|^2}{\|x\|^2-\sigma^2}}\quad &\rm{if }\,\  \|x\|<\sigma,\\
0&\rm{if } \,\ \sigma\leq \| x\| \leq \frac12.
\end{cases}
\end{equation*}
Then $\mathcal K_3+h_3(\cdot)\mathcal I$ with $k(z)=k_\delta(z)$ has no principal eigenvalue for $M>1$, $0<\sigma \ll 1$, $\delta\gg 1$ and $h_3(x)=-1 +Mq(x)$ where $x\in \RR^N$ and $N\ge 3$
(see \cite{ShZh0}). Hence
 $\tilde \lambda_3=\max_{x\in \bar D}h_3(x)=-1+M\max_{x\in \bar D} q(x)=-1+M$. Choosing $M=\frac1{1-\hat q}$, we have $M\hat q=-1+M$, that is $\hat a_3=\tilde \lambda_3$, but
 $a_3(x)=M q(x)$ is not a constant function.
\end{remark}

\section{Effects of Dispersal Rates and the Proof of Theorem \ref{dispersal-rate-effect-thm}}

In this section, we investigate the effects of the dispersal rates
 on the principal spectrum points and the
existence of principal eigenvalues of nonlocal dispersal operators and prove Theorem  \ref{dispersal-rate-effect-thm}.

\begin{proof}[Proof of Theorem \ref{dispersal-rate-effect-thm}]

(1) Assume that $k(\cdot)$ is symmetric. Observe that for any $u(\cdot)\in L^2(D)$,
\begin{align*}
&\int\int_{D\times D}k(y-x)u(x)u(y)dydx-\int_D u^2(x)dx\\
&\leq \int_D \int_D k(y-x)u(y)u(x)dydx-\int_D\int_D k(y-x)dy u^2(x)dx\\
&=\int_D\int_Dk(y-x)(u(y)-u(x))u(x)dydx\\
&=\frac{1}{2}\int\int_{D\times D} k(y-x)(u(y)-u(x))u(x)dydx+\frac{1}{2}\int\int_{D\times D}k(y-x)(u(y)-u(x))u(x)dydx\\
&=\frac{1}{2}\int\int_{D\times D}k(y-x)(u(y)-u(x))u(x)dydx+\frac{1}{2}\int\int_{D\times D}k(y-x)(u(x)-u(y))u(y)dydx\\
&=-\frac{1}{2}\int\int_{D\times D}k(y-x)(u(y)-u(x))^2dydx\\
&\leq 0.
\end{align*}
Then (1) follows from the following facts: $\forall \,\nu_i>0$,
\[
\tilde \lambda_i(\nu_i, a_i)=\sup_{u\in L^2(D), ||u||_{L^2(D)}=1}\left[\nu_i \left(\int_D\int_Dk(y-x)u(y)u(x)dydx-\int_D u^2(x)dx\right)+\int_D a_i(x)u^2(x)dx\right]
\]
in the case $i=1$,
\[
\tilde \lambda_i(\nu_i, a_i)=\sup_{u\in L^2(D), ||u||_{L^2(D)}=1}\left[-\frac{\nu_i}{2}\int\int_{D\times D}k(y-x)(u(y)-u(x))^2dydx+\int_Da_i(x)u^2(x)dx\right]
\]
in the case $i=2$, and
\[
\tilde \lambda_i(\nu_i, a_i)=\sup_{u\in L^2(D), ||u||_{L^2(D)}=1}\left[\nu_i\left(\int_D\int_D \hat k(y-x)u(y)u(x)dydx-\int_D u^2(x)dx\right)+\int_D a_i(x)u^2(x)dx\right]
\]
in the case $i=3$ (see \eqref{K-P-equivalent-eq}).

(2) We prove the case $i=1$. The case $i=3$ can be proved similarly.

Without loss of generality, assume $a_1(x)>0$ for $x\in \bar D$.
 Assume that $\nu_1>0$ is such that $\lambda_1(\nu_1,a_1)$ exists and $\tilde\nu_1>\nu_1$.
 By proposition \ref{iff-prop},  $\lambda_1(\nu_1, a_1)>\max_{x\in \bar D}h_1(x)$,
that is,
\[
\lambda_1(\nu_1, a_1)>\max_{x\in \bar D}(-\nu_1+a_1(x)).
\]

Let $\phi_1(\cdot)$ be a positive principal eigenfunction with $||\phi_1||_{L^2(D)}=1$. Then
\begin{align*}
\lambda_1(\nu_1, a_1)=\nu_1\int\int_{D\times D}k(y-x)\phi_1(y)\phi_1(x)dydx-\nu_1+\int_Da_1(x)
\phi_1^2(x)dx>\max_{x\in\bar D}(-\nu_1+a_1(x)).
\end{align*}
By Proposition \ref{variation-characterization-prop},
\begin{align*}
\tilde\lambda_1(\tilde\nu_1,a_1)&\ge \tilde\nu_1\int\int_{D\times D} k(y-x)\phi_1(y)\phi_1(x)dydx-\tilde\nu_1+\int_D a_1(x)\phi_1^2(x)dx\\
&=\lambda_1(\nu_1, a_1)+(\tilde \nu_1-\nu_1)\int\int_{D\times D}k(y-x)\phi_1(y)\phi_1(x)dydx+\nu_1-\tilde\nu_1\\
&>\max_{x\in\bar D}(-\nu_1+a_1(x))+\nu_1-\tilde \nu_1+(\tilde\nu_1-\nu_1)\int\int_{D\times D}k(y-x)\phi_1(y)\phi_1(x)dydx\\
&> \max_{x\in\bar D}(-\tilde\nu_1+a_1(x)).
\end{align*}
By proposition \ref{iff-prop} again, $\lambda_1(\tilde\nu_1,a_1)$ exists.

(3)
It follows from Theorem 2.1(1) and can also be proved as follows.

 To show $\lambda_i(\nu_i, a_i)$ exists, we only need to show $\tilde \lambda_i(\nu_i, a_i)>\max_{x\in \bar D}h_i(x)$, where $h_i(x)=-\nu_i+a_i(x)$ for $i=1$ and $3$ and $h_i(x)=-\nu_i\int_Dk(y-x)dy+a_i(x)$ for $i=2$.
In the case $i=2$ or $3$,
$\tilde \lambda_i(\nu_i,a_i)\geq \hat a_i$  by theorem \ref{spatial-effect-thm}(4). This implies that
$$
\tilde \lambda_i(\nu_i,a_i)>h_{i,\max}\quad \forall\,\, \nu_i\gg 1.
$$
In the case $i=1$, note that $\lambda_1(1,0)$ exists and
$$
-1<\lambda_1(1,0)<0.
$$
This implies that $\lambda_1(1,\frac{a_1}{\nu_1})$ exists for $\nu_1\gg 1$
and then $\lambda_1(\nu_1,a_1)$ exists for $\nu_1\gg 1$.

(4) On the one hand, we have
$$
\tilde\lambda_i(\nu_i,a_i)\geq h_{i,\max}\geq -\nu_i+a_{i,\max}.
$$

On the other hand, for any $\lambda>a_{i,\max}$, $\lambda \mathcal{I}-a_i(\cdot)\mathcal {I}$ has bounded inverse.
This implies that
$$
a_{i, \max}+\epsilon>\tilde\lambda_i(\nu_i,a_i)\quad \forall\,\, 0<\nu_i\ll 1.
$$
Therefore,
$$
\lim_{\nu_i\to 0}\tilde\lambda_i(\nu_i,a_i)=a_{i,\max}.
$$

(5) We prove the cases $i=1$ and $i=2$. The case $i=3$ can be proved by the similar arguments as in the case $i=2$.

 First of all, we prove the case $i=1$. By Proposition \ref{most-basic-prop},
$$
\tilde\lambda_1(1,0)<0.
$$
Observe that
$$
\tilde\lambda_1(\nu_1,a_1)=\nu_1\tilde\lambda_1\left(1,\frac{a_1}{\nu_1}\right)
\quad {\rm and}\quad
\tilde\lambda_1\left(1,\frac{a_1}{\nu_1}\right)\to \tilde\lambda_1(1,0)
$$
as $\nu_1\to \infty$. It then follows that
$$
\tilde\lambda_1(\nu_1,a_1)\le\frac{\nu_1}{2}\tilde\lambda_1(1,0)\quad \forall \,\,\nu_1\gg 1.
$$
This implies that
$$
\lim_{\nu_1\to\infty}\tilde\lambda_1(\nu_1,a_1)=-\infty.
$$

Second of all, we prove the case $i=2$.
By (3),  $\lambda_2(\nu_2,a_2)$ exists for $\nu_2\gg 1$.
In the following, we assume $\nu_2\gg 1$ such that $\lambda_2(\nu_2,a_2)$ exists. Let $\phi_{2,\nu_2}(x)$ be a positive principal
eigenfunction with $\int_D \phi^2_{2, \nu_2}(x)dx=1$.

Note that
$$
\hat a_2\le \lambda_2(\nu_2,a_2)\le a_{2,\max},
$$
and
$$
\nu_2\int_D\int_D k(y-x)(\phi_{2,\nu_2}(y)-\phi_{2,\nu_2}(x))\phi_{2,\nu_2}(x)dydx+\int_D a_2(x)\phi_{2,\nu_2}^2(x)dx=\lambda_2(\nu_2,a_2).
$$
This implies that
$$
\frac{\nu_2}{2}\int_D\int_D k(y-x)(\phi_{2,\nu_2}(y)-\phi_{2,\nu_2}(x))^2 dydx=\int_D a_2(x)\phi_{2,\nu_2}^2(x)dx-\lambda_2(\nu_2,a_2)\le a_{2,\max}-\hat a_2,
$$
and then
\begin{equation}
\label{nu-2-eq1}
\int_D\int_D k(y-x)(\phi_{2,\nu_2}(y)-\phi_{2,\nu_2}(x))^2 dydx\leq \frac{2(a_{2,\max}-\hat a_2)}{\nu_2}.
\end{equation}

Let $\psi_{2,\nu_2}(x)=\phi_{2,\nu_2}(x)-\hat \phi_{2,\nu_2}$.
Then
\begin{equation*}
\nu_2\int_D\int_D k(y-x)(\phi_{2,\nu_2}(y)-\phi_{2,\nu_2}(x))dydx+\int_D a_2(x)\phi_{2,\nu_2}(x) dx=\int_D a_2(x)(\psi_{2,\nu_2}(x)+\hat \phi_{2,\nu_2})dx,
\end{equation*}
and hence
\begin{equation*}
\lambda_2(\nu_2,a_2)\int_D \phi_{2,\nu_2}(x)dx=\hat \phi_{2,\nu_2}\int_D a_2(x)dx+\int_D a_2(x) \psi_{2,\nu_2}(x)dx.
\end{equation*}
This implies that
\begin{equation}
\label{nu-2-eq2}
\lambda_2(\nu_2,a_2)\hat \phi_{2,\nu_2}=\hat a_2\hat \phi_{2,\nu_2}+\frac1{|D|}\int_D a_2(x)\psi_{2,\nu_2}(x)dx.
\end{equation}

To show $\lambda_2(\nu_2,a_2)\to \hat a_2$ as $\nu_2\to\infty$, we first show that
$ \int_D a_2(x)\psi_{2,\nu_2}(x)dx\to 0$ as $\nu_2\to\infty$.

Note that $\tilde\lambda_2(1,0)=0$ and $\tilde\lambda_2(1,0)$ is the principal eigenvalue of
$\mathcal{K}_2+b_0(\cdot)\mathcal{I}$ with $\phi(\cdot)\equiv 1$ being a principal eigenfunction, where
$$
b_0(x)=-\int_D k(y-x)dy.
$$
Moreover, $\tilde\lambda_2(1,0)$ is also an isolated algebraically simple   eigenvalue of $\mathcal{K}_2+b_0(\cdot)\mathcal{I}$ on $L^2(D)$.

Note also that
\begin{equation}
\label{inequ}
\int_D\Big((-\mathcal{K}_2-b_0\mathcal {I})u\Big)(x)u(x)dx=\frac{1}{2}\int_D\int_D k(y-x)(u(y)-u(x))^2 dydx\geq 0
\end{equation}
for any $u(\cdot)\in L^2(D)$ and $-\mathcal{K}_2-b_0(\cdot)\mathcal{I}$ is a self-adjoint operator on $L^2(D)$.
Then there is a bounded linear operator $A:L^2(D)\to L^2(D)$ such that
\begin{equation}
\label{nu-2-eq3}
\int_D\Big((-\mathcal{K}_2-b_0\mathcal {I})u\Big)(x)u(x)dx=\int_D(Au)(x)(Au)(x)dx\quad \forall\,\, u\in L^2(D).
\end{equation}
Let
$$
E_1={\rm span}\{\phi(\cdot)\},
$$
and
$$
E_2=\{u(\cdot)\in L^2(D)\,|\, \int_D u(x)dx=0\}.
$$
Then
$$
L^2(D)=E_1\oplus E_2
$$
and
$$
\Big(\mathcal{K}_2+b_0(\cdot)\mathcal{I}\Big)(E_2)\subset E_2.
$$
Moreover, $(\mathcal{K}_2+b_0(\cdot)\mathcal{I})|_{E_2}$ is invertible. We claim that
there is $C>0$ such that
\begin{equation}
\label{nu-2-eq4}
\int_D (Au)(x)(Au)(x)dx\geq C\int_D u^2(x)dx\quad \forall\,\, u\in E_2.
\end{equation}
For otherwise, there is $u_n\in E_2$ with $\int_D u_n^2(x)dx=1$ such that
$$
\int_D (Au_n)(x)(Au_n)(x)dx\to 0
$$
as $n\to\infty$. It then follows that $0\in\sigma((\mathcal{K}_2+b_0(\cdot)\mathcal{I})|_{{E_2}})$, a contradiction.
Hence \eqref{nu-2-eq4} holds.

By \eqref{inequ}, \eqref{nu-2-eq3} and \eqref{nu-2-eq4},
for any $\nu_2\gg 1$,
\begin{equation}
\label{nu-2-eq5}
\int_D \psi_{2,\nu_2}^2(x)dx\le \frac{1}{2C} \int_D\int_D k(y-x)(\psi_{2,\nu_2}(y)-\psi_{2,\nu_2}(x))^2dydx.
\end{equation}
Observe that
$$
\int_D\int_D k(y-x)(\phi_{2,\nu_2}(y)-\phi_{2,\nu_2}(x))^2dydx=\int_D\int_D k(y-x)(\psi_{2,\nu_2}(y)-\psi_{2,\nu_2}(x))^2dydx.
$$
This together with \eqref{nu-2-eq1} and \eqref{nu-2-eq5} implies that
$$
\int_D \psi_{2,\nu_2}^2(x) dx\to 0\quad {\rm as}\quad \nu_2\to\infty,
$$
and then
 $$
 \int_D a_2(x)\psi_{2,\nu_2}(x)dx\to 0\quad {\rm as}\quad \nu_2\to\infty.
 $$

Second, assume $\lambda_2(\nu_2,a_2)\not \to\hat a_2$ as $\nu_2\to\infty$. By \eqref{nu-2-eq2},  we must have
 $\hat\phi_{2,\nu_{2,n}}\to 0$ for some sequence $\nu_{2,n}\to \infty$.  This and \eqref{nu-2-eq1} implies that
 \begin{align*}
 \int_D\phi_{2,\nu_{2,n}}^2(x)dx&\leq C_0\int_D\int_D k(y-x)\phi^2_{2,\nu_{2,n}}(x)dydx\\
 &=C_0\int_D \int_D k(y-x)(\phi_{2,\nu_{2,n}}^2(x)-\phi_{2,\nu_{2,n}}(x)\phi_{2,\nu_{2,n}}(y))dydx\\
 &\quad +C_0\int_D\int_D k(y-x)\phi_{2,\nu_{2,n}}(y)\phi_{2,\nu_{2,n}}(x)dydx\\
 &\leq \frac{C_0}{2}\int_D\int_D k(y-x)(\phi_{2,\nu_{2,n}}(y)-\phi_{2,\nu_{2,n}}(x))^2dydx+|D|^2C_0M\hat\phi_{2,\nu_{2,n}}\hat\phi_{2,\nu_{2,n}}\\
 &\leq \frac{C_0(a_{2, \max}-\hat a_2)}{\nu_2}+|D|^2C_0M\hat\phi_{2,\nu_{2,n}}\hat\phi_{2,\nu_{2,n}}
 \end{align*}
where $C_0=(\min_{x\in\bar D}\int_D k(y-x)dy)^{-1}$ and $M=\sup_{x,y\in\bar D}k(y-x)$.
That is
\[
\int_D\phi_{2,\nu_{2,n}}^2(x)dx\to 0 \quad \text{ as } \nu_{2, n}\to\infty.
\]
This is a contradiction. Therefore
$$
\lambda_2(\nu_2,a_2)\to\hat a_2
$$
as $\nu_2\to\infty$.
\end{proof}

\section{Effects of Dispersal Distance and the Proof of Theorem \ref{dispersal-distance-effect-thm}}

In this section, we investigate the effects of the dispersal distance
 on the principal spectrum points and the
existence of principal eigenvalues and prove Theorem  \ref{dispersal-distance-effect-thm}.

\begin{proof}[Proof of Theorem \ref{dispersal-distance-effect-thm}]

(1) As mentioned in Remark \ref{dispersal-distance-effect-rk}, the cases $i=1$  and $3$ are proved in \cite[Theorem 2.6]{KaLoSh1}.
The case $i=2$  can be proved by the similar arguments as in  \cite[Theorem 2.6]{KaLoSh1}.  For completeness, we
provide a proof for the case $i=2$ in the following.

By Proposition \ref{variation-characterization-prop},
$$
\tilde \lambda_i(\nu_i,a_i,\delta)=\sup_{u\in L^2(D),\|u\|_{L^2(D)}=1}\int_D\left[\nu_i\int_D k_\delta(y-x) (u(y)-u(x))dy+a_i(x)u(x)\right] u(x) dx.
$$
On the one hand,
\begin{equation*}
\tilde \lambda_i(\nu_i,a_i,\delta)
\!=\!\sup_{u\in L^2(D),\|u\|_{L^2(D)}=1}\left[-\frac{\nu_i}2\int_D\int_Dk_{\delta}(y-x)(u(y)-u(x))^2dydx+\int_D\int_Da_i(x)u^2(x)dx \right]
\leq a_{i, \max}.
\end{equation*}
On the other hand, assume that $x_0\in\bar D$ is such that $a_i(x_0)=a_{i,\max}$. Then for
any $0<\epsilon<1$, there are $\sigma_0^*>0$ and  $x_0^*\in {\rm
Int} D$ such that $B(x_0^*, \sigma_0^*)\subset \bar D$ and
$$ a_i(x_0)-a_i(x)<\epsilon/2\quad {\rm for}\quad x\in
B(x_0^*, \sigma_0^*).
$$
Let $u_0(\cdot)$ be a smooth
function with ${\rm supp}(u_0(\cdot))\cap D\subset
B(x_0^*, \sigma_0^*)$ and $\|u_0\|_{L^2(D)}=1$.
Then
\begin{align*}
\tilde\lambda_i(\nu_i,a_i,\delta)&\geq
\int_D \left(\nu_i\int_Dk_\delta(y-x)(u_0(y)-u_0(x))dy+a_i(x)u_0(x)\right)u_0(x)dx\\
&\geq
\nu_i\int_D \left(\int_Dk_\delta(y-x)(u_0(y)-u_0(x))dy\right)u_0(x)dx+\left(a_{i,\max}-\frac{\epsilon}{2}\right).
\end{align*}
Note that
$$
\int_D k_\delta(y-x)(u_0(y)-u_0(x))dy \to 0\quad
 \forall \, x\in {\rm Int}(D)
$$
as $\delta\to 0$. And
$$
\left|\int_Dk_\delta(y-x)(u_0(y)-u_0(x))dy\right|\leq 2 \max_{y\in \bar D}|u_0(y)|\quad
\forall\, x\in D.
$$
Hence, there exists $\delta_0>0$, such that for any $\delta<\delta_0$, we have
$$
\left|\nu_i\int_D\left(\int_Dk_\delta(y-x)(u_0(y)-u_0(x))dy\right)u_0(x)dx\right|\leq \frac{\epsilon}{2}
$$
It then follows that
$$
a_{i,\max}\geq \tilde \lambda_i(\nu_i,a_i,\delta)\geq a_{i,\max}-{\epsilon}
$$
 This implies that $\tilde\lambda_i(\nu_i,a_i,\delta)\to a_{i,\max}$ as
$\delta\to 0$.

(2) First, for $i=1$,
$$
\left|\int_D k_\delta(y-x)u(y)dy\right|\leq \|u\|\int_D k_\delta(y-x)dy \to 0
$$
as $\delta\to \infty$ uniformly in $u\in X_1$ with $\|u\|\le 1$. Therefore,
$$\tilde\lambda_1(\nu_1,a_1,\delta)\to \sup\{{\rm Re}\lambda|\lambda\in \sigma((-\nu_1+a_1(\cdot))\mathcal{I})\}=-\nu_1+a_{1,\max}
$$
as $\delta\to\infty$.

For $i=2$,
$$
\left|\int_D k_\delta(y-x)(u(y)-u(x))dy\right|\leq 2\|u\|\int_D k_\delta(y-x)dy \to 0
$$
as $\delta\to\infty$ uniformly in $u\in X_2$ with $\|u\|\le 1$. Hence
$$
\tilde\lambda_2(\nu_2,a_2,\delta)\to \sup\{{\rm Re}\lambda|\lambda\in \sigma(a_2(\cdot)\mathcal{I})\} =a_{2,\max}
$$
as $\delta\to\infty$.

For $i=3$, recall that
$$
\bar\lambda_3(\nu_3,a_3)=\sup\{{\rm Re}\lambda\,|\, \lambda \in\sigma(\nu_2\mathcal{\bar I}+h_3(\cdot)\mathcal{I})\},
$$
where
$$
\mathcal{\bar I}u=\frac{1}{p_1p_2\cdots p_N}\int_0^{p_1}\int_0^{p_2}\cdots\int_0^{p_N}u(x)dx.
$$
We first assume that $a_3(\cdot)$ satisfies the conditions in Remark \ref{spatial-effect-rk} (2). Then by similar arguments as in Theorem
\ref{spatial-effect-thm} (2),
 $\bar\lambda_3(\nu_3,a_3)$ is the principal eigenvalue of $\nu_3\mathcal{\bar I}+h_3(\cdot)\mathcal{I}$. Let
 $\phi_3(\cdot)$ be the positive principal eigenfunction of  $\nu_3\mathcal{\bar I}+h_3(\cdot)\mathcal{I}$ with $\hat \phi_3=\frac{1}{|D|}\int_D\phi_3(x)dx=1$. We then have
$\bar\lambda_3(\nu_3,a_3)>h_{3,\max}$ and
\begin{equation}
\label{del-3-eq0}
\frac{1}{|D|}\int_D \frac{\nu_3\psi_3(x)}{\bar \lambda_3(\nu_3,a_3)+\nu_3-a_3(x)}dx=1,
\end{equation}
where
$$
\psi_3(x)=(\bar\lambda_3(\nu_3,a_3)+\nu_3-a_3(x))\phi_3(x).
$$
Fix $0<\epsilon<\bar\lambda_3(\nu_3,a_3)-h_{i,\max}$. Then
\begin{equation}
\label{del-3-eq1}
\frac{1}{|D|}\int_D \frac{\nu_3\psi_3(x)}{\bar \lambda_3(\nu_3,a_3)-\epsilon+\nu_3-a_3(x)}dx>1.
\end{equation}

Observe that for any ${\bf k}=(k_1,k_2,\cdots,k_N)\in \ZZ^N\setminus\{0\}$,
$$
\int_{\RR^N} \tilde k(z)\cos\Big(\sum_{i=1}^N k_i p_i x_i+\delta \sum_{i=1}^N k_i p_i z_i\Big)dz\to 0,
$$
and
$$
\int_{\RR^N} \tilde k(z)\sin\Big(\sum_{i=1}^N k_i p_i x_i+\delta \sum_{i=1}^N k_i p_i z_i\Big)dz\to 0
$$
as $\delta \to \infty$. This implies that for any $a\in X_3$,
$$
\int_{\RR^N}\tilde k(z)a(x+\delta z)dz\to \hat a
$$
as $\delta\to\infty$ and then
\begin{align*}
\int_{\RR^N}\frac{\nu_3k_\delta(y-x)\psi_3(y)}{\bar\lambda_3(\nu_3,a_3)-\epsilon+\nu_3-a_3(y)}dy&=\int_{\RR^N}\frac{\nu_3\tilde k(z)\psi_3(x+\delta z)}{\bar\lambda_3(\nu_3,a_3)-\epsilon+\nu_3-a_3(x+\delta z)}dz\\
&\to \frac{1}{|D|}\int_D \frac{\nu_3\psi_3(x)}{\bar \lambda_3(\nu_3,a_3)-\epsilon+\nu_3-a_3(x)}dx
\end{align*}
as $\delta\to\infty$ uniformly in $x\in\RR^N$. This together with \eqref{del-3-eq1} implies that
$$
\int_{\RR^N}\frac{\nu_3k_\delta(y-x)\psi_3(y)}{\bar\lambda_3(\nu_3,a_3)-\epsilon+\nu_3-a_3(y)}dy>1\quad \forall\,\, x\in\RR^N,\,\,\, \delta\gg 1.
$$
It then follows that
\begin{equation}
\label{del-3-eq2}
\tilde\lambda_3(\nu_3,a_3,\delta)>\bar\lambda_3(\nu_3,a_3)-\epsilon>h_{i, \max}\quad \forall\,\, \delta \gg 1
\end{equation}
and $\lambda_3(\nu_3,a_3,\delta)$ exists for $\delta\gg 1$.

Now for any $\epsilon>0$,
by \eqref{del-3-eq0},
\begin{equation}
\label{del-3-eq3}
\frac{1}{|D|}\int_D \frac{\nu_3\psi_3(x)}{\bar \lambda_3(\nu_3,a_3)+\epsilon+\nu_3-a_3(x)}dx< 1.
\end{equation}
Then by the similar arguments in the above,
\begin{equation}
\label{del-3-eq4}
\tilde\lambda_3(\nu_3,a_3,\delta)<\bar\lambda_3(\nu_3,a_3)+\epsilon\quad \forall \,\,\delta\gg 1.
\end{equation}
By \eqref{del-3-eq2} and \eqref{del-3-eq4},
$$
\tilde\lambda_3(\nu_3,a_3,\delta)\to \bar\lambda_3(\nu_3,a_3)\quad {\rm as}\quad \delta\to\infty.
$$

Now for general $a_3\in X_3$, and for any $\epsilon>0$, there is $a_{3,\epsilon}\in X_3$ such that
$$
\|a_3-a_{3,\epsilon}\|<\epsilon\quad \forall \, x\in\RR^N,
$$
and $a_{3,\epsilon}(\cdot)$ satisfies the conditions in Remark \ref{spatial-effect-rk} (2).  By Theorem \ref{spatial-effect-thm} (5),
$$
\tilde\lambda_3(\nu_3,a_{3,\epsilon},\delta)-\epsilon\le\tilde\lambda_3(\nu_3,a_3,\delta)\le\tilde\lambda_3(\nu_3,a_{3,\epsilon},\delta)+\epsilon.
$$
By the above arguments,
$$
\bar\lambda_3(\nu_3,a_3)-3\epsilon\le \bar\lambda_3(\nu_3,a_{3,\epsilon})-2\epsilon\le \tilde\lambda_3(\nu_3,a_3,\delta)\le\bar \lambda_3(\nu_3,a_{3,\epsilon})+2\epsilon\le\bar\lambda_3(\nu_3,a_3)+3\epsilon
\quad \forall\,\,\delta\gg 1.
$$
We hence also have
$$
\tilde\lambda_3(\nu_3,a_3,\delta)\to \bar\lambda_3(\nu_3,a_3)\quad {\rm as}\quad \delta\to\infty.
$$

(3) By (1),  for any $\epsilon>0$,
$$
\tilde\lambda_i(\nu_i,a_i,\delta)> a_{i,\max}-\epsilon\quad \forall\,\, 0<\delta\ll 1.
$$
This implies that there is $\delta_0>0$ such that
$$
\tilde\lambda_i(\nu_i,a_i,\delta)>h_{i,\max}\quad \forall   \,\ 0<\delta<\delta_0.
$$
Then by Proposition \ref{variation-characterization-prop},  $\lambda_i(\nu_i,a_i)$ exists for $0<\delta<\delta_0$.
\end{proof}

\section{Asymptotic Dynamics of Two Species Competition System}

In this section, we consider the asymptotic dynamics of the two species competition system \eqref{competition-system-eq} and
prove Theorem \ref{competition-system-thm}
by applying some of the principal spectrum properties developed in previous sections.
Throughout this section, we assume that $k(-z)=k(z)$,  $\tilde \lambda_1(\nu,f(\cdot,0))>0$, $f(x,w)<0$ for $w\gg 1$, and $\p_2 f(x,w)<0$ for $w\ge 0$.

We first present two lemmas.

\begin{lemma}
\label{dirichlet-neumann-bc-lm}
For any given $\nu>0$ and $a\in X_1 (=X_2)$,
$$
\tilde\lambda_1(\nu,a)\le \tilde\lambda_2(\nu,a)
$$
and if $\lambda_1(\nu,a)$ exists, then
$$
\tilde\lambda_1(\nu,a)(=\lambda_1(\nu,a))< \tilde\lambda_2(\nu,a)
$$

\end{lemma}

\begin{proof}
First, assume that $\lambda_1(\nu,a)$ exists. Let $\phi(\cdot)$ be the positive principal
eigenfunction of $\nu\mathcal{K}_1-\nu\mathcal{I}+a(\cdot)\mathcal{I}$ with $\|\phi\|=1$.
Then
$$
\Phi_1(t;\nu,a)\phi= e^{\lambda_1(\nu,a)t}\phi, \quad {\rm and}\quad \Phi_2(t;\nu,a)\phi= e^{\tilde \lambda_2(\nu,a)t}\phi\quad \forall\,\, t>0.
$$
By Proposition \ref{comparison-prop},
$$
\Phi_2(t;\nu,a)\phi\gg \Phi_1(t;\nu,a)\phi\quad \forall\,\, t>0.
$$
This implies that
$$
\tilde \lambda_2(\nu,a)>\lambda_1(\nu,a).
$$

In general, by Lemma \ref{technical-lm} and Theorem \ref{spatial-effect-thm} (2), for any $\epsilon>0$,  there is $a_\epsilon\in X_1$ such that
$\lambda_1(\nu,a_\epsilon)$ exists and
$$
a_\epsilon(x)-\epsilon\le a(x)\le a_\epsilon(x)+\epsilon.
$$
By the above arguments,
$$
\tilde\lambda_2(\nu,a_\epsilon)>\lambda_1(\nu,a_\epsilon).
$$
Observe that
$$
\tilde\lambda_2(\nu,a)\ge \tilde\lambda_2(\nu,a_\epsilon)-\epsilon\quad {\rm and}\quad \lambda_1(\nu,a_\epsilon)\geq\tilde\lambda_1(\nu,a)-\epsilon.
$$
Hence
$$
\tilde\lambda_2(\nu,a)\ge \tilde\lambda_1(\nu,a)-2\epsilon.
$$
Letting $\epsilon\to 0$, we have
$$
\tilde\lambda_2(\nu,a)\ge\tilde\lambda_1(\nu,a).
$$
\end{proof}

Consider
\begin{equation}
\label{single-eq1}
u_t=\nu\left[\int_D k(y-x)u(t,y)dy-u(t,x)\right]+u(t,x)g(x,u(t,x)),\quad x\in\bar D
\end{equation}
and
\begin{equation}
\label{single-eq2}
v_t=\nu \int_D k(y-x)[v(t,y)-v(t,x)]dy+v(t,x)g(x,v(t,x)),\quad x\in\bar D,
\end{equation}
where $g$ is a $C^1$ function, $g(x,w)<0$ for $w\gg 1$, and $\p_2 g(x,w)<0$ for $w\ge 0$.

\begin{lemma}
\label{single-eq-lm}
\begin{itemize}
\item[(1)] If $\lambda_1(\nu, g(\cdot,0))>0$, then there is $u^*\in X_1^{++}$ such that $u=u^*$ is a stationary solution of
\eqref{single-eq1} and for any solution $u(t,x)$ of \eqref{single-eq1} with $u(0,\cdot)\in X_1^{+}\setminus\{0\}$,
$u(t,\cdot)\to u^*(\cdot)$ in $X_1$.

\item[(2)]  If $\lambda_2(\nu, g(\cdot,0))>0$, then there is $v^*\in X_2^{++}$ such that $v=v^*$ is a stationary solution of
\eqref{single-eq2} and for any solution $v(t,x)$ of \eqref{single-eq2} with $v(0,\cdot)\in X_2^{+}\setminus\{0\}$,
$v(t,\cdot)\to v^*(\cdot)$ in $X_2$.
\end{itemize}
\end{lemma}

\begin{proof}
It follows from \cite[Theorem E]{RaSh}.
\end{proof}

\begin{proof}[Proof of Theorem \ref{competition-system-thm}]
(1) By $\tilde \lambda_1(\nu,f(\cdot,0))>0$ and Lemma \ref{dirichlet-neumann-bc-lm}, we  have $\tilde \lambda_2(\nu,f(\cdot,0))>0$.
Then by Lemma \ref{single-eq-lm},
 there are $u^*\in X_1^{++}$ and $v^*\in X_2^{++}$ such that
$(u^*,0)$ and $(0,v^*)$ are stationary solutions of \eqref{competition-system-eq}. Moreover,
for any $(u_0,v_0)\in X_1^+\times X_2^+$ with $u_0\not = 0$ and $v_0=0$ (resp. $u_0=0$ and $v_0\not = 0$),
$(u(t,\cdot;u_0,v_0),v(t,\cdot;u_0,v_0))\to (u^*(\cdot),0)$ (resp.
$(u(t,\cdot;u_0,v_0),v(t,\cdot;u_0,v_0))\to (0,v^*(\cdot))$) as $t\to\infty$.

(2)
Observe that
\begin{equation}
\label{competition-eq1}
\nu\left[\int_D k(y-x)u^*(y)dy-u^*(x)\right]+f(x,u^*(x))u^*(x)=0,\quad x\in\bar D.
\end{equation}
This implies that $\lambda_1(\nu,f(\cdot,u^*(\cdot)))$ exists and $\lambda_1(\nu,f(\cdot,u^*(\cdot)))=0$.
By Lemma \ref{dirichlet-neumann-bc-lm},  we have
$$
\tilde\lambda_2(\nu,f(\cdot,u^*(\cdot)))>0.
$$

By Lemma \ref{technical-lm}, there are $\epsilon>0$ and  $a\in X_1$ such that
$\lambda_2(\nu,a)$ exists,
$$
a(x)\le f(x,u^*(x))-\epsilon,\quad \lambda_2(\nu,a)>0,
$$
and
$$
\tilde\lambda_2(\nu, f(\cdot,u^*(\cdot)+\epsilon))>0.
$$

Let $\phi(\cdot)$ be the positive eigenfunction of $\nu\mathcal{K}_2-\nu b(\cdot)\mathcal{I}+a(\cdot)\mathcal{I}$ with $\|\phi\|=1$,
where $b(x)=\int_D k(y-x)dy$.
Let
$$
u_\delta(x)=u^*(x)+\delta^2\quad {\rm  and}\quad v_\delta(x)=\delta\phi(x).
$$
Then
\begin{align*}
0&=\nu\left[\int_D k(y-x)u^*(y)dy-u^*(x)\right]+u^*(x)f(x,u^*(x))\\
&=\nu\left[\int_D k(y-x)u_\delta(y)dy-u_\delta(x)\right]+u_\delta(x) f(x, u_\delta(x)+v_\delta(x))\\
&\quad +\nu\delta^2 \left(1-\int_D k(y-x)dy\right)-\delta^2 f(x,u^*(x))\\
&\quad +u_\delta \left[f(x,u^*(x))-f(x,u_\delta(x)+v_\delta(x))\right]\\
&\ge \nu\left[\int_D k(y-x)u_\delta(y)dy-u_\delta(x)\right]+u_\delta(x) f(x, u_\delta(x)+v_\delta(x))
\end{align*}
for $0<\delta\ll 1$, and
\begin{align*}
0&\le \lambda_2(\nu,a)v_\delta(x)\\
&=\nu\int_D k(y-x)[v_\delta(y)-v_\delta(x)]dy+a(x)v_\delta(x)\\
&\le \nu \int_D k(y-x)[v_\delta(y)-v_\delta(x)]dy +[f(x,u^*(x))-\epsilon]v_\delta(x)\\
&=\nu \int_D k(y-x)[v_\delta(y)-v_\delta(x)]dy + v_\delta(x) f(x,u_\delta(x)+v_\delta(x))\\
&\quad +v_\delta(x)\left [f(x,u^*(x))-f(x,u_\delta(x)+v_\delta(x))-\epsilon\right]\\
&\le \nu \int_D k(y-x)[v_\delta(y)-v_\delta(x)]dy + v_\delta(x) f(x,u_\delta(x)+v_\delta(x))
\end{align*}
for $0<\delta\ll 1$. It then follows that for $0<\delta\ll 1$, $(u_\delta(x),v_\delta(x))$ is a super-solution of
\eqref{competition-system-eq}.
By Proposition \ref{comparison-system-prop},
\begin{equation}
\label{monotone-eq}
(u(t_2,\cdot;u_\delta,v_\delta),v(t_2,\cdot;u_\delta,v_\delta))\leq_2 (u(t_1,\cdot;u_\delta,v_\delta),v(t_1,\cdot;u_\delta,v_\delta))\quad \forall\,\, 0<t_1<t_2.
\end{equation}
Let
$$
(u_\delta^{**}(x),v_\delta^{**}(x))=\lim_{t\to\infty} (u(t,x;u_\delta,v_\delta),v(t,x;u_\delta,v_\delta)) \quad \forall\,\, x\in\bar D
$$
(this pointwise limit exists because of \eqref{monotone-eq}).

We claim that $(u_\delta^{**}(\cdot),v_\delta^{**}(\cdot))=(0,v^*(\cdot))$. Observe that  $u_\delta^{**}(\cdot)$ and $v_\delta^{**}(\cdot)$ are semi-continuous and
$(u_\delta^{**}(\cdot), v_\delta^{**}(\cdot))$ satisfies that
\begin{equation}
\label{competition-system-star-eq}
\begin{cases}
\nu[\int_D k(y-x)u_\delta^{**}(y)dy-u_\delta^{**}(x)]+u_\delta^{**}(x)f(x,u_\delta^{**}(x)+v_\delta^{**}(x))=0,\quad x\in\bar D,\cr
\nu \int_D k(y-x)[v_\delta^{**}(y)-v_\delta^{**}(x)]dy+v_\delta^{**}(x) f(x,u_\delta^{**}(x)+v_\delta^{**}(x))=0,\quad x\in\bar D
\end{cases}
\end{equation}
(see the arguments in \cite[Theorem A]{HeNgSh}). Multiplying the first equation in \eqref{competition-system-star-eq} by $v_\delta^{**}(x)$,
second equation by $u_\delta^{**}(x)$, and integrating over $D$, we have
$$
\int_D u_\delta^{**}(x)v_\delta^{**}(x)dx=\int_D \left(\int_D k(y-x)dy\right)u_\delta^{**}(x)v_\delta^{**}(x)dx.
$$
This together with $v_\delta^{**}(x)\ge \delta\phi(x)>0$  implies that
$$
\left[1-\int_D k(y-x)dy\right]u_\delta^{**}(x)=0\quad \forall\,\, x\in\bar D.
$$
Note that $\int_D k(y-x)dy<1$ for $x$ near $\p D$. This together with the first equation in \eqref{competition-system-star-eq}
implies that $u_\delta^{**}(x)=0$ for all $x\in\bar D$. We then must have $v_\delta^{**}(x)=v^*(x)$ for all $x\in \bar D$.
Moreover, by \eqref{monotone-eq} and Dini's theorem,
\begin{equation}
\label{limit-behavior-eq}
\lim_{t\to\infty} (u(t,\cdot;u_\delta,v_\delta),v(t,\cdot;u_\delta,v_\delta))=(0,v^*(\cdot)) \quad {\rm in}\quad X_1\times X_2.
\end{equation}

Now, for any $(u_0,v_0)\in (X_1^+\setminus\{0\})\times (X_2^+\setminus\{0\})$, there is $M_0>0$ such that
$$
(u_0,v_0)\le_2 (M,0).
$$
Then by Proposition \ref{comparison-system-prop},
$$
(u(t,\cdot;u_0,v_0),v(t,\cdot;u_0,v_0))\le_2(u(t,\cdot;M,0),v(t,\cdot;M,0))\quad \forall\,\, t>0.
$$
Since $(u(t,\cdot;M,0),v(t,\cdot;M,0))\to (u^*(\cdot),0)$ in $X_1\times X_2$ for $0<\delta\ll 1$, there is $T>0$ such that
$$
(u(t,\cdot;u_0,v_0),v(t,\cdot;u_0,v_0))\le_2 (u_\delta(\cdot),0)\quad \forall\,\, t\ge T.
$$
Then $v(t,\cdot;u_0,v_0)$ satisfies
$$
v_t(t,x)\ge \nu\int_Dk(y-x)[v(t,y)-v(t,x)]dy+ v(t,x)f(x,u^*(x)+\epsilon+v(t,x))
$$
for $t\ge T$. Note that $\tilde\lambda_2(\nu,f(\cdot,u^*(\cdot)+\epsilon))>0$. By Lemma \ref{single-eq-lm}, for $0<\delta\ll 1$,
there is $\tilde T\ge T$ such that
$$
v(t,\cdot;u_0,v_0)\ge v_\delta(\cdot)\quad \forall\,\, t\ge 0.
$$
We then have
$$
(u(t+\tilde T,\cdot;u_0,v_0),v(t+\tilde T,\cdot;u_0,v_0))\le_2 (u(t,\cdot;u_\delta,v_\delta),v(t,\cdot;u_\delta,v_\delta))\quad \forall  \,\, t\ge 0.
$$
By \eqref{limit-behavior-eq},
$$
\lim_{t\to\infty}(u(t,\cdot;u_0,v_0),v(t,\cdot;u_0,v_0))=(0,v^*(\cdot)).
$$
The theorem is thus proved.
\end{proof}

\end{document}